\documentclass[sts,
	reqno,
]{imsart}

\usepackage{amsthm, amssymb, amsmath}
\usepackage{mathrsfs, mathtools}
\usepackage{euscript}
\usepackage{color}
\usepackage{tikz}
\usetikzlibrary{arrows}
\usepackage[utf8]{inputenc}
\usepackage[shortlabels]{enumitem}
\usepackage{url}
\usepackage{etex}
\usepackage{alltt}
\usepackage{memhfixc}
\usepackage{bbold}%
\usepackage{cases}
\usepackage[noadjust]{cite}
\definecolor{labelkey}{rgb}{0.0, 0.8, 0.3}
\usepackage[top=2.5cm,bottom=2.5cm,left=2.5cm,right=2.5cm,includeheadfoot,asymmetric]{geometry}
\usepackage{algorithm}
\usepackage{algpseudocode}
\usepackage{xfrac}
\usepackage{wrapfig}

\usepackage{custom}

\numberwithin{equation}{section}
\usepackage{hyperref}
\hypersetup{
  colorlinks = true,
  urlcolor = blue, 
  linkcolor = blue,
  citecolor = red}

\declaretheorem[name=Remark, style=remark]{rmk}
\declaretheorem[name=Theorem]{thm}

\renewcommand{\abstractname}{Abstract}

\begin{document}

\begin{frontmatter}

	\title{SVGD as a kernelized Wasserstein gradient flow of the chi-squared divergence}
	\runtitle{SVGD as a kernelized Wasserstein gradient flow}
	\author{Sinho Chewi \hfill schewi@mit.edu \\
	Thibaut Le Gouic \hfill tlegouic@mit.edu \\
	Chen Lu \hfill chenl819@mit.edu \\
	Tyler Maunu \hfill maunut@mit.edu \\
	Philippe Rigollet \hfill rigollet@mit.edu}



	\address{{Department of Mathematics} \\
		{Massachusetts Institute of Technology}\\
		{77 Massachusetts Avenue,}\\
		{Cambridge, MA 02139-4307, USA}
	}
	


\runauthor{Chewi et al.}

\begin{abstract}
Stein Variational Gradient Descent (SVGD), a popular sampling algorithm, is often described as the kernelized gradient flow for the Kullback-Leibler divergence in the geometry of optimal transport. We introduce a new perspective on SVGD that instead views SVGD as the (kernelized) gradient flow of the chi-squared divergence which, we show, exhibits a strong form of uniform exponential ergodicity  under conditions as weak as a Poincar\'e inequality. This perspective leads us to propose an alternative to SVGD, called Laplacian Adjusted Wasserstein Gradient Descent (LAWGD), that can be implemented from the spectral decomposition of the Laplacian operator associated with the target density. We show that LAWGD exhibits strong convergence guarantees and good practical performance.
\end{abstract}



\end{frontmatter}

\section{Introduction}

The seminal paper of Jordan, Kinderlehrer, and Otto~\cite{jordan1998variational} has profoundly reshaped our understanding of sampling algorithms. What is now commonly known as the \emph{JKO scheme} interprets the evolution of marginal distributions of a Langevin diffusion as a gradient flow  of a Kullback-Leibler (KL) divergence over the Wasserstein space of probability measures. This optimization perspective on Markov Chain Monte Carlo (MCMC)  has not only renewed our understanding of algorithms based on Langevin diffusions~\cite{Dal17, Ber18, cheng2018langevin, Wib18, durmus2019lmcconvex, vempala2019langevin}, but has also fueled the discovery of new MCMC algorithms inspired by the diverse and powerful optimization toolbox~\cite{martin2012newtonmcmc, simsekli2016quasinewtonlangevin, cheng2017underdamped, Ber18, hsieh2018mirrored, Wib18, ma2019there, Wib19prox, chewi2020exponential, DalRiou2020,  zhang2020mirror}.

The Unadjusted Langevin Algorithm (ULA)~\cite{dalalyan2017theoretical, durmus2017nonasymptoticlangevin} is the most common discretization of the Wasserstein gradient flow for the KL divergence, but it is unclear whether it is the most effective one. In fact, ULA is asymptotically biased, which results in slow convergence and often requires ad-hoc adjustments~\cite{dwivedi2019log}. To overcome this limitation, various methods that track the Wasserstein gradient flow more closely have been recently developed~\cite{Ber18, Wib18, salim2020proximal}.

Let $F$ denote a functional over the Wasserstein space of distributions. The Wasserstein gradient flow of $F$ may be described as the deterministic and time-inhomogeneous Markov process ${(X_t)}_{t\ge 0}$ started at a random variable $X_0 \sim \mu_0$ and evolving according to  $\dot{X}_t= -[\nabla_{W_2}F(\mu_t)](X_t)$, where $\mu_t$ denotes the distribution of $X_t$. Here $[\nabla_{W_2}F(\mu)](\cdot): \R^d \to \R^d$ is the Wasserstein gradient of $F$ at $\mu$. If $F(\mu)=\KL{\mu}$, where $\pi \propto \e^{-V}$ is a given target distribution on $\R^d$, it is known~\cite{ambrosio2008gradient, villani2009ot, santambrogio2017euclidean} that $\nabla_{W_2}F(\mu)= \nabla \ln(\D \mu/\D \pi)$.
Therefore, a natural discretization of the Wasserstein gradient flow with step size $h > 0$, albeit one that cannot be implemented since it depends on the distribution $\mu_{t}$ of $X_{t}$, is:
$$
X_{t+1} =X_t - h \nabla \ln\big(\frac{\D \mu_t}{\D \pi}(X_t)\big), \qquad t=0,1, 2, \ldots\,.
$$

While $\mu_t$ can, in principle, be estimated by evolving a large number of particles $X_t^{[1]}, \ldots, X_t^{[N]}$, estimation of $\mu_t$ is hindered by the curse of dimensionality and this approach still faces significant computational challenges despite attempts to improve the original JKO scheme~\cite{salim2020proximal, wang2020informationnewton}.

A major advance in this direction was achieved by allowing for \emph{approximate}  Wasserstein gradients. More specifically, Stein Variational Gradient Descent (SVGD), recently proposed by \cite{liu2016stein} (see Section~\ref{sec:svgd} for more details), consists in replacing $\nabla_{W_2}F(\mu)$ by its image $\cK_\mu \nabla_{W_2}F(\mu)$ under the integral operator  $\cK_\mu: L^2(\mu) \to L^2(\mu)$ associated to a chosen kernel $K:\R^d\times \R^d \to \R$ and defined by $\cK_\mu f(x) := \int K(x,y) f(y) \, \D \mu(y)$ for $f \in L^2 (\mu)$.
This leads to the following process:
\begin{equation}
\label{eq:SVGD_p}
\tag{$\msf{SVGD_{p}}$}
\dot{X}_t= -[\cK_{\mu_t}\nabla_{W_2}F(\mu_t)](X_t)\,,
\end{equation}
where we apply the integral operator $\mc K_{\mu_t}$ individually to each coordinate of the Wasserstein gradient.
In turn, this \emph{kernelization trick} overcomes most of the above computational bottleneck. 
Building on this perspective,~\cite{duncan2019geometrysvgd} introduced a new geometry, different from the Wasserstein geometry and which they call the \emph{Stein geometry}, in which~\eqref{eq:SVGD_p} becomes the gradient flow of the KL divergence. However, despite this recent advance, the theoretical properties of SVGD as a sampling algorithm as well as guidelines for the choice of the kernel $K$ are still largely unexplored.

In this work, we  revisit the above view of SVGD as a kernelized gradient flow of the KL divergence over Wasserstein space that was put forward in~\cite{liu2017stein}.

\medskip

\textbf{Our contributions}. 
We introduce, in Section~\ref{sec:svgd_as_chi_sq}, a new perspective on SVGD by viewing it as kernelized gradient flow of the chi-squared divergence rather than the KL divergence. This perspective is fruitful in two ways.
First, it uses a single integral operator $\cK_\pi$---as opposed to~\eqref{eq:SVGD_p}, which requires a family of integral operators $\cK_\mu$, $\mu\ll\pi$---providing a conceptually clear guideline for choosing $K$, namely: $K$ should be chosen to make $\cK_\pi$ approximately equal to the identity operator.
Second, under the idealized choice $\cK_\pi = \id$, we show that this gradient flow converges exponentially fast in KL divergence as soon as the target distribution $\pi$ satisfies a Poincar\'e inequality. In fact, our results are stronger than exponential convergence and they highlight \emph{strong uniform ergodicity}: the gradient flow forgets the initial distribution after a finite time that is at most half of the Poincar\'e constant. 
To establish this exponential convergence under a relatively weak condition (Poincar\'e inequality), we employ the following technique. While the gradient flow aims at minimizing the chi-squared divergence by following the curve in Wasserstein space with steepest descent, we do not track its progress with the objective function itself, the chi-squared divergence, but instead we track it with the KL divergence. This is in a sense dual to argument employed in~\cite{chewi2020exponential}, where the chi-squared divergence is used to track the progress of a gradient flow on the KL divergence. A more standard analysis relying on \L{}ojasiewicz inequalities also yields rates of convergence on the chi-squared divergence under stronger assumptions such as a log-Sobolev inequality, and log-concavity.
These results establish the first finite-time theoretical guarantees for SVGD in an idealized setting. 

\begin{wrapfigure}[15]{r}{.31\textwidth}
\vspace{-0.7cm}
    \centering
    \includegraphics[width = 0.31\textwidth]{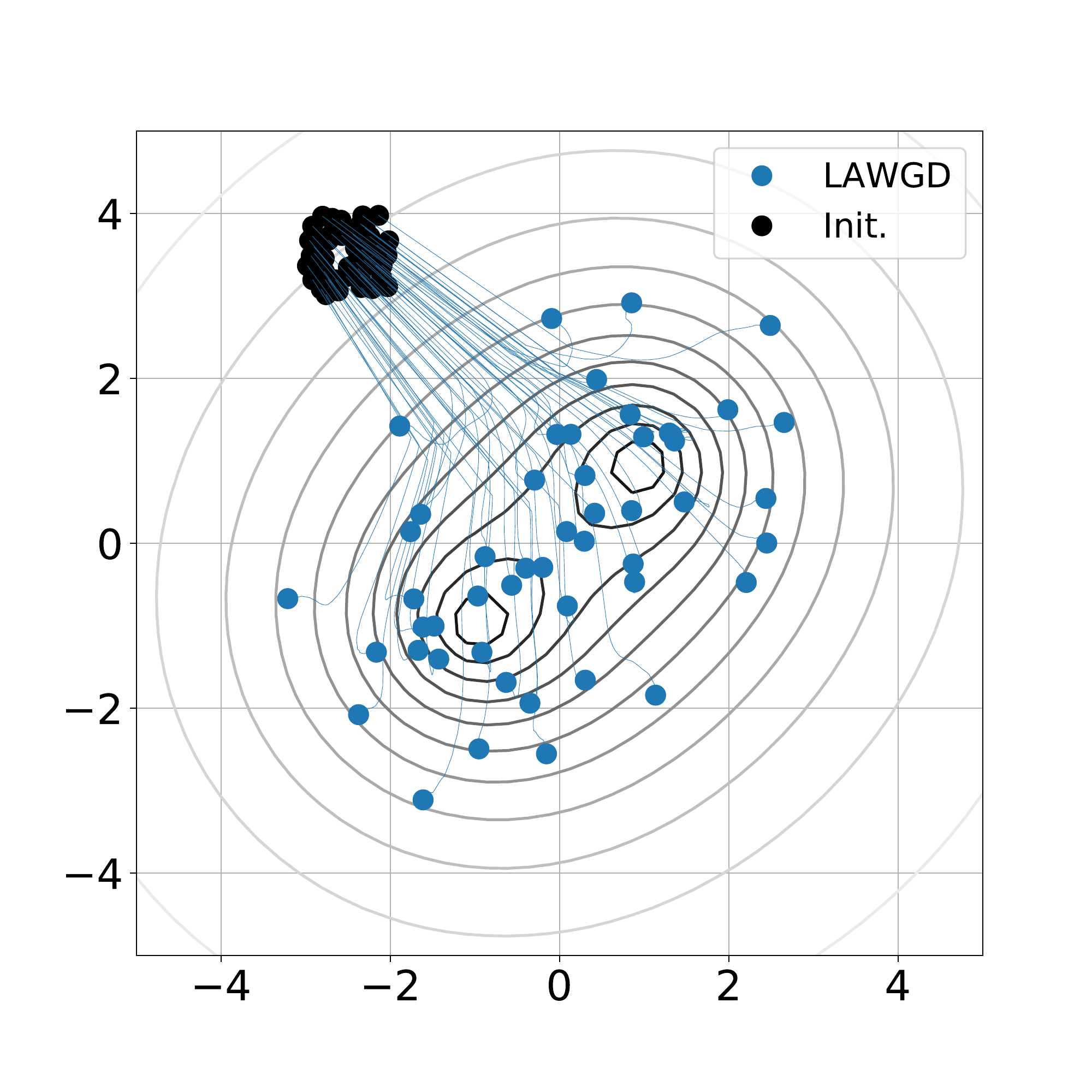}
    \caption{Sampling from a mixture of two 2D Gaussians with \ref{eq:msvgd}. See Appendix~\ref{append:numerics}.}
    \label{fig:gaussmix2d}
\end{wrapfigure}

Beyond providing a better understanding of SVGD, our novel perspective is instrumental in the development of a new sampling algorithm, which we call Laplacian Adjusted Wasserstein Gradient Descent~\eqref{eq:msvgd} and present in Section~\ref{sec:KGD}. Although LAWGD is challenging to implement in high dimensions, we show that it possesses a striking theoretical property: assuming that the target distribution $\pi$ satisfies a Poincar\'e inequality, LAWGD converges exponentially fast, with \emph{no} dependence on the Poincar\'e constant. This scale invariance has been recently demonstrated for the Newton-Langevin diffusion~\cite{chewi2020exponential}, but under the additional assumption that $\pi$ is log-concave. A successful implementation of LAWGD hinges on the spectral decomposition of a certain differential operator which is within reach of modern PDE solvers. As a proof of concept, we show that LAWGD performs well in one or two dimensions using a na\"{i}ve finite differences method and leave the question of applying more sophisticated numerical solvers open for future research.

\medskip

\textbf{Related work}.
Since its introduction in \cite{liu2016stein}, a number of variants of SVGD have been considered. They include a stochastic version~\cite{li2019stochastic}, a version that approximates the Newton direction in Wasserstein space~\cite{detommaso2018stein}, a version that uses matrix kernels~\cite{wang2019stein}, an accelerated version~\cite{liu2019understanding}, and a hybrid with Langevin~\cite{zhang2018stochastic}. Several works have studied theoretical properties of SVGD, including its interpretation as a gradient flow under a modified geometry~\cite{liu2017stein,duncan2019geometrysvgd}, and its asymptotic convergence~\cite{lu2019scaling}.

\medskip

\textbf{Notation}. In this paper, all probability measures are assumed to have densities w.r.t.\ Lebesgue measure; therefore, we frequently abuse notation by identifying a probability measure with its Lebesgue density. For a differentiable kernel $K : \R^d\times\R^d\to\R$, we denote by $\nabla_1 K :\R^d\times\R^d\to\R^d$ (resp.\ $\nabla_2 K$) the gradient of the kernel w.r.t.\ the first (resp.\ second) argument. When describing particle algorithms, we use a subscript to denote the time index and brackets to denote the particle index, i.e., $X_t^{[i]}$ refers to the $i$th particle at time (or iteration number) $t$.

\section{SVGD as a kernelized Wasserstein gradient flow}
\label{sec:svgd}
\subsection{Wasserstein gradient flows}\label{sec:wgf}

In this section, we review the theory of gradient flows on the space $\mc P_{2,\rm ac}(\R^d)$ of probability measures absolutely continuous w.r.t.\ Lebesgue measure and possessing a finite second moment, equipped with the $2$-Wasserstein metric $W_2$. We refer readers to~\cite{villani2003topics, santambrogio2015ot, santambrogio2017euclidean} for introductory treatments of optimal transport, and to~\cite{ambrosio2008gradient, villani2009ot} for detailed treatments of Wasserstein gradient flows.


Let $F : \mc P_{2,\rm ac}(\R^d) \to \R\cup\{\infty\}$ be a functional defined on Wasserstein space. We say that a curve ${(\mu_t)}_{t\ge 0}$ of probability measures is a \emph{Wasserstein gradient flow} for the functional $F$ if it satisfies
\begin{align}\label{eq:wgf}
    \partial_t \mu_t
    &= \divergence\bigl(\mu_t \nabla_{W_2} F(\mu_t)\bigr)
\end{align}
in a weak sense. Here, $\nabla_{W_2} F(\mu) := \nabla \delta F(\mu)$ is the Wasserstein gradient of the functional $F$ at $\mu$, where $\delta F(\mu) : \R^d\to\R$ is the \emph{first variation} of $F$ at $\mu$, defined by
\begin{align*}
    \lim_{\varepsilon\to 0} \frac{F(\mu+\varepsilon \xi)-F(\mu)}{\varepsilon}
    = \int \delta F(\mu) \, \D \xi, \qquad\text{for all}~\xi~\text{with}~\int \D \xi = 0,
\end{align*}
and $\nabla$ denotes the usual (Euclidean) gradient. Hence, the Wasserstein gradient, at each $\mu \in  \mc P_{2,\rm ac}(\R^d)$, is a map from $\R^d$ to $\R^d$.

Using the \emph{continuity equation}, we can give an Eulerian interpretation to the evolution equation~\eqref{eq:wgf} (see~\cite[\S 4]{santambrogio2015ot} and~\cite[\S 8]{ambrosio2008gradient}). Given a family of vector fields ${(v_t)}_{t\ge 0}$, let ${(X_t)}_{t\ge 0}$ be a curve in $\R^d$ with random initial point $X_0 \sim \mu_0$, and such that ${(X_t)}_{t\ge 0}$ is an integral curve of the vector fields ${(v_t)}_{t\ge 0}$, that is, $\dot X_t = v_t(X_t)$. If we let $\mu_t$ denote the law of $X_t$, then ${(\mu_t)}_{t\ge 0}$ evolves according to the \emph{continuity equation}
\begin{align}\label{eq:continuity_eq}
    \partial_t \mu_t
    &= -\divergence(\mu_t v_t).
\end{align}
Comparing~\eqref{eq:wgf} and~\eqref{eq:continuity_eq}, we see that~\eqref{eq:wgf} describes the evolution of the marginal law ${(\mu_t)}_{t\ge 0}$ of the curve ${(X_t)}_{t\ge 0}$ with $X_0 \sim \mu_0$ and $\dot X_t = -[\nabla_{W_2} F(\mu_t)](X_t)$.

Wasserstein calculus provides the following (formal) calculation rule: the Wasserstein gradient flow ${(\mu_t)}_{t\ge 0}$ for the functional $F$ dissipates $F$ at the rate $\partial_t F(\mu_t) = -\E_{\mu_t}[\norm{\nabla_{W_2}F(\mu_t)}^2]$. More generally, for a curve ${(\mu_t)}_{t\ge 0}$ evolving according to the continuity equation~\eqref{eq:continuity_eq}, the time-derivative of $F$ is given by $\partial_t F(\mu_t) = \E_{\mu_t}\langle \nabla_{W_2} F(\mu_t), v_t \rangle$.

In this paper, we are primarily concerned with two functionals: the Kullback-Leibler (KL) divergence $D_{\rm KL}(\cdot \mmid \pi)$, and the chi-squared divergence $\chi^2(\cdot \mmid \pi)$ (see, e.g., \cite{tsybakov2009nonparametric}). It is a standard exercise~\cite{ambrosio2008gradient, santambrogio2015ot} to check that Wasserstein gradients of these functionals are, respectively,
\begin{align}\label{eq:w2_grads}
    \bigl(\nabla_{W_2} D_{\rm KL}(\cdot \mmid \pi)\bigr)(\mu) = \nabla \ln \frac{\D \mu}{\D \pi}, \qquad \bigl(\nabla_{W_2} \chi^2(\cdot \mmid \pi)\bigr)(\mu) = 2\nabla \frac{\D \mu}{\D \pi}.
\end{align}

\subsection{SVGD as a kernelized gradient flow of the KL divergence}\label{sec:svgd_as_kl}



SVGD\footnote{Throughout this paper, we call SVGD the generalization of the original method of~\cite{liu2016stein, liu2017stein} that was introduced in~\cite{wang2019stein}.} is achieved  by replacing the Wasserstein gradient $\nabla \ln (\D\mu_t/\D \pi)$ of the KL divergence with $\cK_{\mu_t} \nabla \ln (\D\mu_t/\D \pi)$, leading to the particle evolution equation~\eqref{eq:SVGD_p}.

Recalling that $\pi \propto e^{-V}$, we get
\begin{align}
    \cK_{\mu_t} \nabla \ln \frac{\D\mu_t}{\D\pi}(x)
    &:= \int K(x,\cdot) \nabla \ln \frac{\D\mu_t}{\D\pi} \, \D \mu_t
    = \int K(x,\cdot) \nabla V \, \D \mu_t - \int \nabla_2 K(x,\cdot) \, \D \mu_t\,,\label{eq:svgd_gradient}
\end{align}
where, in the second identity, we used integration by parts. This expression shows that rather than having to estimate the distribution $\mu_t$, it is sufficient to estimate the expectation $\int \nabla_2 K(x,\cdot) \, \D \mu_t$. This is the key to the computational tractability of SVGD. Indeed, the kernelized gradient flow can implemented by drawing $N$ particles $X_0^{[1]},\dotsc,X_0^{[N]} \simiid \mu_0$ and following the coupled dynamics
\begin{align*}
    \dot X_t^{[i]}
    &= -\cK_{\mu_t} \nabla \ln \frac{\D\mu_t}{\D\pi}(X_t^{[i]})
    = -\int K(X_t^{[i]},\cdot) \nabla V \, \D \mu_t + \int \nabla_2 K(X_t^{[i]},\cdot) \, \D \mu_t, \qquad i \in [N].
\end{align*}
With this, we can simply estimate the expectation with respect to $\mu_t$ with an average over all particles. Discretizing the resulting process in time, we obtain the SVGD algorithm:
\begin{align}\label{eq:svgd_alg}
    X_{t+1}^{[i]}
    &= X_t^{[i]} - \frac{h}{N} \sum_{j=1}^N K(X_t^{[i]}, X_t^{[j]}) \nabla V(X_t^{[j]}) + \frac{h}{N} \sum_{j=1}^N \nabla_2 K(X_t^{[i]}, X_t^{[j]}), \qquad i \in [N].
\end{align}

\subsection{SVGD as a kernelized gradient flow of the chi-squared divergence}\label{sec:svgd_as_chi_sq}

Recall from Section~\ref{sec:wgf} that by the continuity equation, the particle evolution equation~\eqref{eq:SVGD_p} translates into the following PDE that describes the evolution of the distribution \(\mu_t\) of $X_t$:
\begin{equation}
\tag{$\msf{SVGD_d}$}
    \label{eq:SVGD_d}
      \partial_t \mu_t
    = \divergence\bigl(\mu_t \cK_{\mu_t} \nabla \ln \frac{\D \mu_t}{\D \pi}\bigr).
\end{equation}

We make the simple observation that
\begin{align*}
    \cK_{\mu_t} \nabla \ln \frac{\D\mu_t}{\D\pi}(x)
    &= \int K(x,y) \nabla \ln \frac{\D\mu_t}{\D\pi}(y) \, \D \mu_t(y)
    = \int K(x,y) \nabla \frac{\D\mu_t}{\D\pi}(y) \, \D \pi(y)
    = \cK_\pi \nabla \frac{\D\mu_t}{\D\pi}(x).
\end{align*}
Thus, the continuous-dynamics of SVGD, as given in~\eqref{eq:SVGD_d}, can equivalently be expressed as
\begin{align}\label{eq:svgd}\tag{$\msf{SVGD}$}
    \partial_t \mu_t
    &= \divergence\bigl(\mu_t \cK_\pi \nabla \frac{\D\mu_t}{\D\pi}\bigr).
\end{align}

To interpret this equation, we recall that the Wasserstein gradient of the chi-squared divergence $\chi^2(\cdot \mmid \pi)$ at $\mu$ is $2\nabla(\D\mu/\D\pi)$ (by~\eqref{eq:w2_grads}), so
the gradient flow for the chi-squared divergence is
\begin{align}\label{eq:csf}\tag{$\msf{CSF}$}
    \partial_t \mu_t
    &= 2 \divergence\bigl(\mu_t \nabla \frac{\D\mu_t}{\D\pi}\bigr).
\end{align}
Comparing~\eqref{eq:svgd} and~\eqref{eq:csf}, we see that (up to a factor of $2$), \ref{eq:svgd} can be understood as the flow obtained by replacing the gradient of the chi-squared divergence, $\nabla (\D \mu/\D \pi)$, by $\cK_\pi \nabla  (\D \mu/\D \pi)$.

Although~\eqref{eq:SVGD_d} and~\eqref{eq:svgd} are equivalent ways of expressing the same dynamics, the formulation of~\eqref{eq:svgd} presents a significant advantage: it involves a kernel integral operator $\cK_{\pi}$ that does not change with time and depends only on the target distribution $\pi$.

\section{Chi-squared gradient flow}\label{sec:csf}

In this section, study the idealized case where $\cK_\pi$ taken to be the identity operator. In this case, \eqref{eq:svgd} reduces to the gradient flow \ref{eq:csf}.
%
The existence, uniqueness, and regularity of this flow are studied in~\cite{ohta2011generalizedentropies, ohta2013generalizedentropiesii} and~\cite[Theorem 11.2.1]{ambrosio2008gradient}.

The rate of convergence of the gradient flow of the KL divergence is closely related to two functional inequalities: the Poincar\'e inequality controls the rate of exponential convergence in chi-squared divergence (\cite[Theorem 4.4]{pavliotis2014stochastic}, ~\cite{chewi2020exponential}) while a log-Sobolev inequality characterizes the rate of exponential convergence of the KL divergence~\cite[Theorem 5.2.1]{bakry2014markov}.
In this section, we show that these inequalities also guarantee exponential rates of convergence of \ref{eq:csf}.

Recall that $\pi$ satisfies a \emph{Poincar\'e inequality} with constant $C_{\msf P}$ if
\begin{align}\label{eq:poincare}\tag{$\msf{P}$}
\var_\pi f \le C_{\msf P}\E_\pi[\norm{\nabla f}^2], \qquad\text{for all locally Lipschitz}~f \in L^2(\pi),
\end{align}
while $\pi$ satisfies a \emph{log-Sobolev inequality} with constant $C_{\msf{LSI}}$
\begin{align}\label{eq:lsi}\tag{$\msf{LSI}$}
\on{ent}_\pi(f^2) := \E_\pi[f^2 \ln(f^2)] - \E_\pi[f^2] \ln \E_\pi[f^2] \le 2C_{\msf{LSI}} \E_\pi[\norm{\nabla f}^2]
\end{align}
for all locally Lipschitz $f$ for which $\on{ent}_\pi(f^2) < \infty$.

We briefly review some facts regarding the strength of these assumptions.
It is standard that the log-Sobolev inequality is stronger than the Poincar\'e inequality: \eqref{eq:lsi} implies~\eqref{eq:poincare} with constant $C_{\msf P} \le C_{\msf{LSI}}$. In turn, if $\pi$ is \emph{$\alpha$-strongly log-concave}, i.e.\ $\nabla^2 V \succeq \alpha I_d$, then it implies the validity of~\eqref{eq:lsi} with $C_{\msf{LSI}} \le 1/\alpha$, and thus a Poincar\'e inequality holds too. However, a Poincar\'e inequality is in general much weaker than strong log-concavity. For instance, if $\lambda_\pi^2$ denotes the largest eigenvalue of the covariance matrix of $\pi$, then it is currently known that $\pi$ satisfies a Poincar\'e inequality as soon as it is log-concave, with $C_{\msf P} \le C(d) \lambda_\pi^2$, where $C(d)$ is a dimensional constant~\cite{bobkovIsoperimetricAnalyticInequalities1999, alonso2015kls, leevempala2017poincare}, and the well-known \emph{Kannan-Lov\'asz-Simonovitz \emph{(KLS)} conjecture}~\cite{kls1995} asserts that $C(d)$ does not actually depend on the dimension.

Our first result shows that a Poincar\'e inequality  suffices to establish exponential decay of the KL divergence along \ref{eq:csf}. In fact, we establish a remarkable property, which we call \emph{strong uniform ergodicity}: under a Poincar\'e inequality,  \ref{eq:csf}  forgets its initial distribution after a time of no more than $C_{\msf P}/2$. Uniform ergodicity is central in the theory of Markov processes~\cite[Ch.\ 16]{MeyTwe09} but is often limited to compact state spaces. Moreover, this theory largely focuses on total variation, so the distance from the initial distribution to the target distribution is trivially bounded by $1$.

\begin{thm}\label{thm:csf_kl_conv}
Assume that $\pi$ satisfies a Poincar\'e inequality~\eqref{eq:poincare} with constant $C_{\msf P} > 0$ and let ${(\mu_t)}_{t\ge 0}$ denote the law of \ref{eq:csf}. Assume that $\chi^2(\mu_0 \mmid \pi) < \infty$. Then,
\begin{align}
\label{eq:expCSF}
    D_{\rm KL}(\mu_t \mmid \pi) \le D_{\rm KL}(\mu_0 \mmid \pi) \, e^{-\frac{2t}{C_{\msf P}}}\,, \qquad \forall\ t \ge 0.
\end{align}
In fact, a stronger convergence result holds:
\begin{align}
\label{eq:exp+csf}
D_{\rm KL}(\mu_t\mmid\pi)\le \bigl(D_{\rm KL}(\mu_0 \mmid \pi) \wedge 2\bigr) \, \e^{-\frac{2t}{C_{\msf P}}} \, \,, \qquad \forall\ t \ge \frac{C_{\msf P}}{2}.
\end{align}
\end{thm}
\begin{proof}
Given the Wasserstein gradients~\eqref{eq:w2_grads} in Section~\ref{sec:wgf}, we get that
${(\mu_t)}_{t\ge0}$ satisfies
\begin{align*}
\partial_tD_{\rm KL}(\mu_t\mmid\pi)
&=-2\E_{\mu_t}\bigl\langle\nabla\ln \frac{\D\mu_t}{\D\pi},\nabla \frac{\D\mu_t}{\D\pi}\bigr\rangle
=-2\E_{\pi}\bigl[\bigl\lVert \nabla \frac{\D\mu_t}{\D\pi} \bigr\rVert^2\bigr].
\end{align*}
Applying the Poincar\'e inequality~\eqref{eq:poincare} with $f=\D\mu_t/\D\pi-1$, we get
\[
\partial_tD_{\rm KL}(\mu_t\mmid\pi)\le -\frac{2}{C_{\msf P}}\chi^2(\mu_t\mmid\pi)\le-\frac{2}{C_{\msf P}}D_{\rm KL}(\mu_t\mmid\pi)\,,
\]
where, in the last inequality, we use the fact that $D_{\rm KL}(\cdot \mmid\pi)\le \chi^2(\cdot\mmid\pi)$ (see~\cite[\S 2.4]{tsybakov2009nonparametric}).
The bound~\eqref{eq:expCSF} follows by applying Gr\"onwall's inequality.

To prove~\eqref{eq:exp+csf}, we use the stronger inequality $D_{\rm KL}(\cdot\mmid\pi) \le \ln[1+\chi^2(\cdot\mmid\pi)]$ (see~\cite[\S 2.4]{tsybakov2009nonparametric}). Our differential inequality now reads:
    \begin{align*}
        \partial_t D_{\rm KL}(\mu_t \mmid \pi)
        \le -\frac{2}{C_{\msf{P}}} \bigl( e^{D_{\rm KL}(\mu_t \mmid \pi)} - 1\bigr) 
         \iff   \partial_t \psi\big(D_{\rm KL}(\mu_t \mmid \pi)\big)
        \le -\frac{2}{C_{\msf{P}}}\psi\big(D_{\rm KL}(\mu_t \mmid \pi)\big)\,,
    \end{align*}
    where $\psi(x)=1-e^{-x} \le 1$.
Gr\"onwall's inequality now yields
    \begin{align*}
    \psi\big(D_{\rm KL}(\mu_t \mmid \pi)\big) \le e^{-\frac{2t}{C_{\sf P}}} \psi\big(D_{\rm KL}(\mu_0 \mmid \pi)\big) \le e^{-\frac{2t}{C_{\sf P}}}.
    \end{align*}
    Note that $x \le 2\psi(x)$ whenever $\psi(x) \le 1/e$. Thus, if  $t \ge C_{\sf P}/2$, we get $\psi\big(D_{\rm KL}(\mu_t \mmid \pi)\big) \le e^{-1}$ so
    \[
    D_{\rm KL}(\mu_t \mmid \pi)\le 2 \psi\big(D_{\rm KL}(\mu_t \mmid \pi)\big) \le e^{-\frac{2t}{C_{\sf P}}}\,,
    \]
which, together with~\eqref{eq:expCSF}, completes the proof of~\eqref{eq:exp+csf}.
\end{proof}

\begin{rmk}
    In~\cite{chewi2020exponential}, it was observed that the chi-squared divergence decays exponentially fast along the gradient flow ${(\mu_t)}_{t\ge 0}$ for the KL divergence,  provided that $\pi$ satisfies a Poincar\'e inequality. This observation is made precise and more general in~\cite{matthes2009fourthorder} where it is noted that the \emph{gradient flow of a functional $\mf U$ dissipates a different functional $\mf V$ at the same rate that the gradient flow of $\mf V$ dissipates the functional $\mf U$}.  A similar method is used to study the thin film equation in~\cite{carrillo2002thinfilm} and~\cite[\S 5]{carlen2011functional}.
\end{rmk}

Since we are studying the gradient flow of the chi-squared divergence, it is natural to ask whether \ref{eq:csf} converges to $\pi$ in chi-squared divergence as well. In the next results, we show quantitative decay of the chi-squared divergence along the gradient flow under a Poincar\'e inequality~\eqref{eq:poincare}, but we obtain only a polynomial rate of decay. However, if we additionally assume either that $\pi$ is log-concave or that it satisfies a log-Sobolev inequality~\eqref{eq:lsi}, then we obtain exponential decay of the chi-squared divergence along \ref{eq:csf}.

\begin{thm}\label{thm:csf_chi_conv_poincare}
Suppose that $\pi$ satisfies a Poincar\'e inequality~\eqref{eq:poincare}.
Then, provided $\chi^2(\mu_0 \mmid \pi) < \infty$, the law ${(\mu_t)}_{t\ge 0}$ of \ref{eq:csf} satisfies
\begin{align*}
    \chi^2(\mu_t \mmid \pi) \le \chi^2(\mu_0\mmid \pi) \wedge \bigl( \frac{9C_{\msf P}}{8t} \bigr)^2.
\end{align*}
If we further assume that $\pi$ is log-concave, then
\begin{align*}
    \chi^2(\mu_t \mmid \pi)
    &\le \chi^2(\mu_0 \mmid \pi) \,  \e^{-\frac{t}{2C_{\msf P}}}.
\end{align*}
\end{thm}
\begin{proof}
    The proof is deferred to Appendix~\ref{sec:csf_proofs}.
\end{proof}

Under the stronger assumption~\eqref{eq:lsi}, we can show strong uniform ergodicity as  in Theorem~\ref{thm:csf_kl_conv}.

\begin{thm}\label{thm:csf_conv_lsi}
Assume that $\pi$ satisfies a log-Sobolev inequality~\eqref{eq:lsi}.
Let ${(\mu_t)}_{t\ge 0}$ denote the law of \ref{eq:csf}, and assume that $\chi^2(\mu_0 \mmid \pi) < \infty$.
Then, for all $t\ge 7C_{\msf{LSI}}$,
\begin{align*}
    \chi^2(\mu_t \mmid \pi)
    &\le \bigl(\chi^2(\mu_0 \mmid \pi) \wedge 2\bigr) \, \e^{-\frac{t}{9C_{\msf{LSI}}}}.
\end{align*}
\end{thm}
\begin{proof}
    The proof is deferred to Appendix~\ref{sec:csf_proofs}.
\end{proof}

Convergence in chi-squared divergence was studied in recent works such as~\cite{cao2019renyi, vempala2019langevin, chewi2020exponential}. From standard comparisons between information divergences (see~\cite[\S 2.4]{tsybakov2009nonparametric}), it implies convergence in total variation distance, Hellinger distance, and KL divergence. Moreover, recent works have shown that the Poincar\'e inequality~\eqref{eq:poincare} yields transportation-cost inequalities which bound the $2$-Wasserstein distance by powers of the chi-squared divergence~\cite{ding2015quadratictransport, ledoux2018remarks, chewi2020exponential, liu2020quadratictransport}, so we obtain convergence in the $2$-Wasserstein distance as well. In particular, we mention that~\cite{chewi2020exponential} uses the chi-squared gradient flow~\eqref{eq:csf} to prove a transportation-cost inequality.

\section{Laplacian Adjusted Wasserstein Gradient Descent (LAWGD)}
\label{sec:KGD}

While the previous section leads to a better understanding of the convergence properties of SVGD in the case that $\cK_\pi$ is the identity operator, it is still unclear how to choose the kernel $K$ to approach this idealized setup. For \ref{eq:svgd} with a general kernel $K$, the calculation rules of Section~\ref{sec:wgf} together with the method of the previous section yield the formula
\begin{align*}
    \partial_t D_{\rm KL}(\mu_t \mmid \pi)
    &= -\E_\pi\bigl\langle \nabla \frac{\D\mu_t}{\D\pi}, \mc K_\pi \nabla \frac{\D\mu_t}{\D\pi} \bigr\rangle,
\end{align*}
for the dissipation of the KL divergence along \ref{eq:svgd}. From this, a natural way to proceed is to seek an inequality of the form
\begin{align}\label{eq:kernel_poincare}
    \E_\pi\langle f, \mc K_\pi f \rangle \gtrsim \E_\pi[f^2], \qquad\text{for all locally Lipschitz}~f \in L^2(\pi).
\end{align}
Applying this inequality to each coordinate of $\nabla (\D \mu_t/\D \pi)$ separately and using a Poincar\'e inequality would then allow us to conclude as in the proof of Theorem~\ref{thm:csf_kl_conv}. The inequality~\eqref{eq:kernel_poincare} can be interpreted as a positive lower bound on the smallest eigenvalue of the operator $\mc K_\pi$.
However, this approach is doomed to fail; under mild conditions on the kernel $K$, it is a standard fact that the eigenvalues of $\mc K_\pi$ form a sequence converging to $0$, so no such spectral gap can hold.\footnote{It is enough that $K$ is a symmetric kernel with $K \in L^2(\pi\otimes \pi)$, and that $\pi$ is not discrete (so that $L^2(\pi)$ is infinite-dimensional); see~\cite[Appendix A.6]{bakry2014markov}.}

This suggests that any approach which seeks to prove finite-time convergence results for \ref{eq:svgd} in the spirit of Theorem~\ref{thm:csf_kl_conv} must exploit finer properties of the eigenspaces of the operator $\mc K_\pi$. Motivated by this observation, we develop a new algorithm called Laplacian Adjusted Wasserstein Gradient Descent \eqref{eq:msvgd} in which the kernel $K$ is chosen carefully so that $\cK_\pi=\cL^{-1}$ is the inverse
of the generator of the Langevin diffusion that has $\pi$ as invariant measure.

More precisely, the starting point for our approach is the following integration-by-parts formula, which is a crucial component of the theory of Markov semigroups~\cite{bakry2014markov}:
\begin{equation}\label{eq:int_by_parts}
    \E_\pi \langle \nabla f, \nabla g \rangle = \E_\pi[f \cL g], \qquad\text{for all locally Lipschitz}~f,g \in L^2(\pi),
\end{equation}
where $\cL := -\Delta + \langle \nabla V, \nabla \cdot\rangle$. The operator $\cL$
is the (negative) generator of the standard Langevin diffusion with stationary distribution $\pi$~\cite[\S 4.5]{pavliotis2014stochastic}. We refer readers to Appendix~\ref{sec:spectral} for background on the spectral theory of $\ms L$.

In order to use~\eqref{eq:int_by_parts}, we replace $-\mc K_\pi \nabla(\D\mu_t/\D\pi)$ by the vector field $-\nabla \mc K_\pi(\D\mu_t/\D\pi)$. The new dynamics 
follow the evolution equation
\begin{align}\label{eq:msvgd}\tag{$\msf{LAWGD}$}
    \partial_t \mu_t
    &= \divergence\bigl( \mu_t \nabla \cK_\pi \frac{\D\mu_t}{\D\pi}\bigr).
\end{align}
The vector field in the above continuity equation may also be written
\begin{align*}
    -\nabla \cK_\pi \frac{\D\mu_t}{\D\pi}(x)
    &= -\int \nabla_1 K(x, \cdot) \, \frac{\D\mu_t}{\D\pi} \, \D \pi
    = -\int \nabla_1 K(x, \cdot) \, \D \mu_t.
\end{align*}

Replacing $\mu_t$ by an empirical average over particles and discretizing the process in time, we again obtain an implementable algorithm, which we give as Algorithm~\ref{ALG:MSVGD}.





A careful inspection of Algorithm~\ref{ALG:MSVGD} reveals that the update equation for the particles in Algorithm~\ref{ALG:MSVGD} does not involve the potential $V$ directly, unlike the SVGD algorithm~\eqref{eq:svgd_alg}; thus, the kernel for \ref{eq:msvgd} must contain all the information about $V$. 

Our choice for the kernel $K$ is guided by the following observation (based on~\eqref{eq:int_by_parts}):
  \begin{align*}
        \partial_t D_{\rm KL}(\mu_t \mmid \pi)
        &= -\E_\pi\bigl\langle \nabla \frac{\D\mu_t}{\D\pi}, \nabla \cK_\pi \frac{\D\mu_t}{\D\pi} \bigr\rangle
        = -\E_\pi\bigl[\frac{\D\mu_t}{\D\pi} \cL\cK_\pi \frac{\D\mu_t}{\D\pi}\bigr].
    \end{align*}
As a result, we choose $K$  to ensure  that $\mc K_\pi=\cL^{-1}$. This choice yields
  \begin{align}
  \label{eq:L-1}
        \partial_t D_{\rm KL}(\mu_t \mmid \pi)
      = -\E_\pi\bigl[\bigl( \frac{\D\mu_t}{\D\pi} - 1\bigr)^2\bigr] = -\chi^2(\mu_t\mmid \pi).
    \end{align}

\begin{wrapfigure}[11]{r}{.5\textwidth}
\vspace{-0.8cm}
\begin{minipage}{0.5\textwidth}
\begin{algorithm}[H]
  \caption{LAWGD}\label{ALG:MSVGD}
  \begin{algorithmic}[1]
    \Procedure{LAWGD}{$\sK_{\cL}, \mu_0$}
      \State draw $N$ particles $X_0^{[1]},\dotsc,X_0^{[N]} \simiid \mu_0$
      \For{$t = 1, \ldots, T-1$}
      \For{$i=1,\dotsc,N$}
      \State $X_{t+1}^{[i]} \gets X_t^{[i]} - \frac{h}{N} \sum_{j=1}^N \nabla_1 \sK_{\cL}(X_t^{[i]}, X_t^{[j]})$
      \EndFor
      \EndFor
      \State \textbf{return } $X_T^{[1]},\dotsc,X_T^{[N]}$
   \EndProcedure
  \end{algorithmic}
\end{algorithm}
\end{minipage}
\end{wrapfigure}

It remains to see which kernel $K$ implements $\mc K_\pi=\cL^{-1}$. To that end, assume that $\cL$ has a discrete spectrum and let $(\lambda_i,\phi_i)$, $i=0,1,2,\dotsc$ be its  eigenvalue-eigenfunction pairs where $\lambda_j$s are arranged in nondecreasing order. Assume further that $\lambda_1 > 0$ (which amounts to a Poincar\'e inequality; see Appendix~\ref{sec:spectral})
and define the following \emph{spectral kernel}:
\begin{equation}\label{eq:KGD_kern}
\sK_{\cL}(x,y) = \sum_{i=1}^\infty \frac{ \phi_i(x) \phi_i(y)}{\lambda_i}
\end{equation}
We now show that this choice of kernel endows \ref{eq:msvgd} with a remarkable property: it converges to the target distribution exponentially fast, with a rate which has no dependence on the Poincar\'e constant.
Moreover, akin to \ref{eq:csf}---see~\eqref{eq:exp+csf}---it also also exhibit strong uniform ergodicity.


\begin{thm}\label{thm:msvgd_conv}
Assume that $\cL$ has a discrete spectrum and that $\pi$ satisfies a Poincar\'e inequality~\eqref{eq:poincare} with some finite constant.
Let ${(\mu_t)}_{t\ge 0}$ be the law of \ref{eq:msvgd} with the kernel described above.
Then, for all $t\ge 1$,
\begin{align*}
    D_{\rm KL}(\mu_t \mmid \pi) \le \bigl( D_{\rm KL}(\mu_0 \mmid \pi) \wedge 2\bigr) \, e^{-t} \,.
\end{align*}
\end{thm}
\begin{proof}
    In light of~\eqref{eq:L-1}, the proof is identical to that of Theorem~\ref{thm:csf_kl_conv}.
\end{proof}

The convergence rate in Theorem~\ref{thm:msvgd_conv} has no dependence on the target measure. This \emph{scale-invariant convergence} also appears in~\cite{chewi2020exponential}, where it is shown for the Newton-Langevin diffusion with a strictly log-concave target measure $\pi$. In Theorem~\ref{thm:msvgd_conv}, we obtain similar guarantees under the much weaker assumption of a Poincar\'e inequality; indeed, there are many examples of non-log-concave distributions which satisfy a Poincar\'e inequality~\cite{vempala2019langevin}.

\section{Experiments}

\begin{wrapfigure}[14]{r}{.34\textwidth}
\vspace{-0.6cm}
    \centering
    \includegraphics[width = 0.31\textwidth]{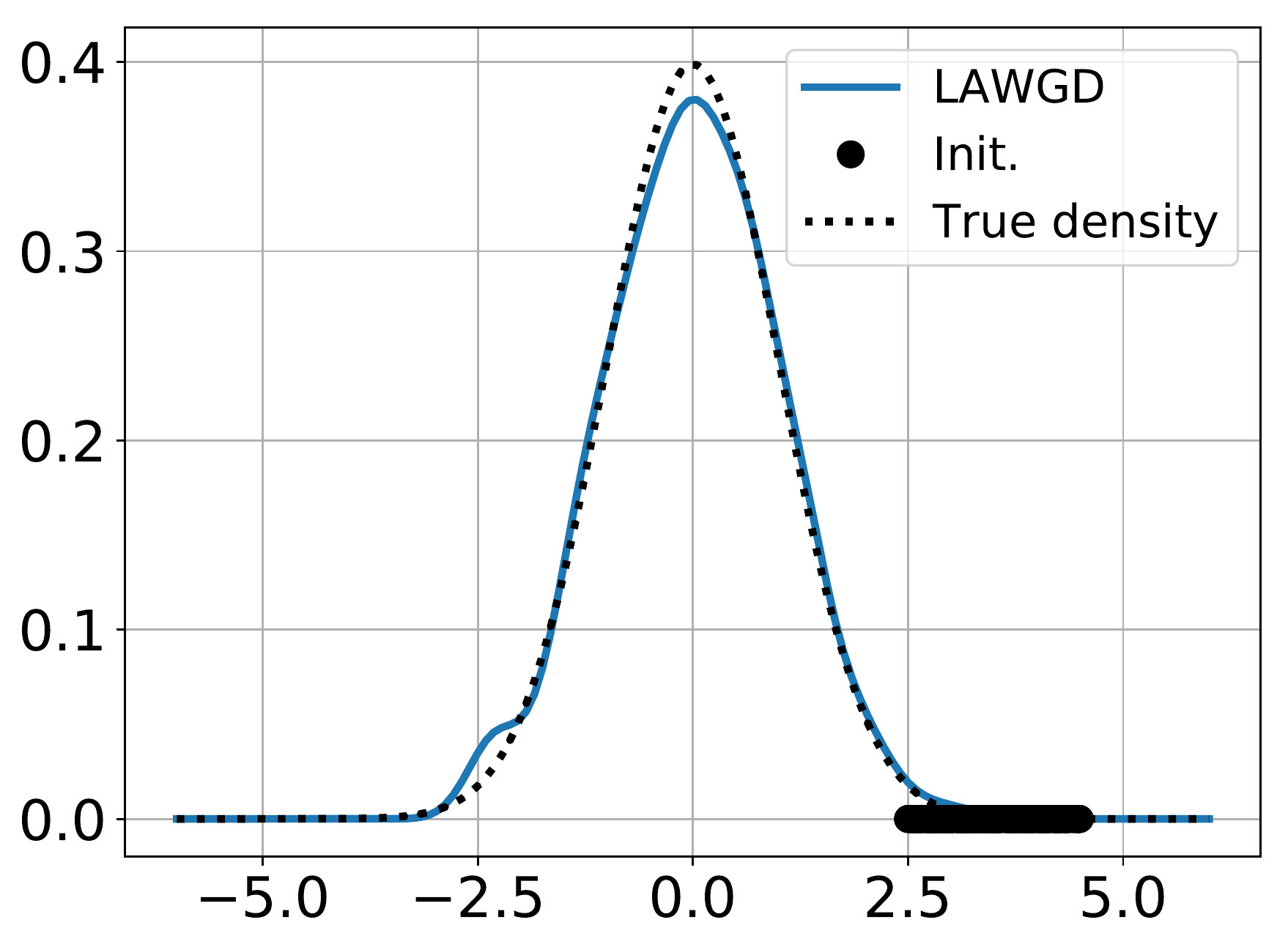}
    \vspace{-0.2cm}
    \caption{Samples from the standard Gaussian distribution generated by \ref{eq:msvgd}, with kernel approximated by Hermite polynomials. For details, see Appendix~\ref{append:numerics}.}
    \label{fig:gaussian}
\end{wrapfigure}

To implement Algorithm~\ref{ALG:MSVGD}, we numerically approximate the kernel $K=\sK_{\cL}$ given in~\eqref{eq:KGD_kern}. When $\pi$ is the standard Gaussian distribution on $\R$, the eigendecomposition of the operator $\ms L$ in~\eqref{eq:int_by_parts} is known explicitly in terms of the Hermite polynomials~\cite[\S 2.7.1]{bakry2014markov}, and we approximate the kernel via a truncated sum: $\hat K(x,y) = \sum_{i=1}^k \lambda^{-1}_i \phi_i(x) \phi_i(y)$ (Figure~\ref{fig:gaussian}) involving the smallest eigenvalues of $\ms L$.

In the general case, we implement a basic finite difference (FD) method to approximate the eigenvalues and eigenfunctions of $\ms L$. We obtain better numerical results by first transforming the operator $\ms L$ into the Schr\"odinger operator $\ms L_{\msf S} := -\Delta + V_{\msf S}$, where $V_{\msf S} := \frac{1}{4}\norm{\nabla V}^2 - \frac{1}{2}\Delta V$. If $\phi_{\msf S}$ is an eigenfunction of $\ms L_{\msf S}$ with eigenvalue $\lambda$ (normalized such that $\int \phi_{\msf S}^2 = 1$), then $\phi := \e^{V/2} \phi_{\msf S}$ is an eigenfunction of $L$ also with eigenvalue $\lambda$ (and normalized such that $\int \phi^2 \, \D \pi = 1$); see~\cite[\S 1.15.7]{bakry2014markov}.

\begin{wrapfigure}[15]{r}{.31\textwidth}
\vspace{-0.9cm}
    \centering
    \includegraphics[width = 0.31\textwidth]{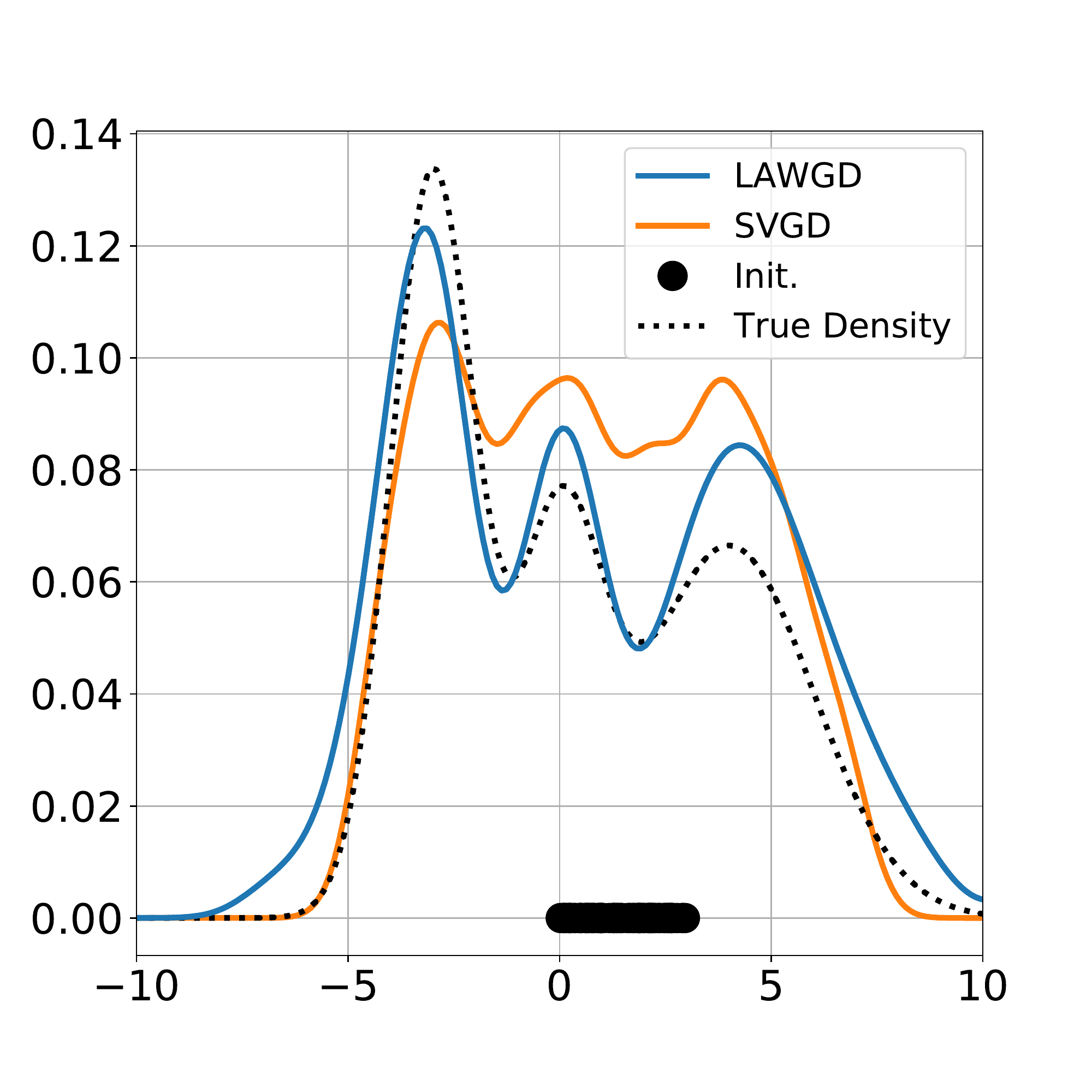}
    \vspace{-0.9cm}
    \caption{\ref{eq:msvgd} and SVGD run with constant step size for a mixture of three Gaussians. Both kernel density estimators use the same bandwidth. }
    \label{fig:gaussmix3}
\end{wrapfigure}

On a grid of points (with spacing $\varepsilon$), if we replace the Laplacian with the FD operator $\Delta_\varepsilon f(x) := \{f(x-\varepsilon) + f(x+\varepsilon) - 2f(x)\}/\varepsilon^2$ (in 1D), then the FD Schr\"odinger operator $\ms L_{{\msf S},\varepsilon} := -\Delta_{\varepsilon} + V_{\msf S}$ can be represented as a sparse matrix, and its eigenvalues and (unit) eigenvectors are found with standard linear algebra solvers.

When the potential $V$ is known only up to an additive constant, then the approximate eigenfunctions produced by this method are not normalized correctly; instead, they satisfy $\norm{\phi}_{L^2(\pi)} = C$ for some constant $C$ (which is the same for each eigenfunction). In turn, this causes the kernel $K$ in \ref{eq:msvgd} to be off by a multiplicative constant. For implementation purposes, however, this constant is absorbed in the step size of Algorithm~\ref{ALG:MSVGD}. We also note that the eigenfunctions are differentiated using a FD approximation.



To demonstrate, we sample from a mixture of three Gaussians: $\frac{2}{5} \cN(-3, 1) + \frac{1}{5} \cN(0, 1) + \frac{2}{5} \cN(4, 2) $. We compare \ref{eq:msvgd} with SVGD using the RBF kernel and median-based bandwidth as in~\cite{liu2016stein}.
We approximate the eigenfunctions and eigenvalues using a finite difference scheme, on 256 grid points evenly spaced between $-14$ and $14$. Constant step sizes for~\ref{eq:msvgd} and SVGD are tuned and the algorithms are run for 5000 iterations, and the samples are initialized to be uniform on $[1, 4]$. The results are displayed in Figure~\ref{fig:gaussmix3}. All 256 discrete eigenfunctions and eigenvalues are used.


\section{Open questions}

We conclude this paper with some interesting open questions. The introduction of the chi-squared divergence as an objective function allows us to obtain both theoretical insights about SVGD and a new algorithm, \ref{eq:msvgd}. This perspective opens the possibility of identifying other functionals defined over Wasserstein space and that yield gradient flows which are amenable to mathematical analysis and efficient computation. Towards this goal, an intriguing direction is to develop alternative methods, besides kernelization, which provide effective implementations of Wasserstein gradient flows. Finally, we note that \ref{eq:msvgd} provides a hitherto unexplored connection between sampling and computing the spectral decomposition of the Schr\"odinger operator, the latter of which has been intensively studied in numerical PDEs. We hope our work further stimulates research at the intersection of these communities.

\medskip

\noindent\textbf{Acknowledgments}. \\
Philippe Rigollet was supported by NSF awards IIS-1838071, DMS-1712596, DMS-TRIPODS-1740751, and ONR
grant N00014-17- 1-2147.
Sinho Chewi and Austin J.\ Stromme were supported by the Department of
Defense (DoD) through the National Defense Science \& Engineering Graduate Fellowship (NDSEG)
Program.
Thibaut Le Gouic was supported by ONR grant N00014-17-1-2147 and NSF IIS-1838071.

\appendix

\section{Review of spectral theory} \label{sec:spectral}

In this paper, we consider elliptic differential operators of the form $\cL = -\Delta + \langle \nabla V, \nabla \cdot \rangle$, where $V$ is a continuously differentiable potential. In this section, we provide a brief review of the spectral theory of these operators, and we refer to~\cite[\S 6.5]{evans2010pde} for a standard treatment.

The operator $\cL$ (when suitably interpreted) is a linear operator defined on a domain $\cD\subset L^2(\pi)$.
For any locally Lipschitz function $f \in L^2(\pi)$, integration by parts shows that
\begin{align*}
    \E_\pi[f\cL f]
    &= \E_\pi[\norm{\nabla f}^2].
\end{align*}
Therefore, $\cL$ has a non-negative spectrum.
Also, we have $\cL 1=0$, so that $0$ is always an eigenvalue of $\cL$.
We say that $\cL$ has a \emph{discrete spectrum} if it has a countable sequence of eigenvalues $0 = \lambda_0 \le \lambda_1 \le \lambda_2 \le \lambda_3 \le \cdots $ and corresponding eigenfunctions ${(\phi_i)}_{i=1}^\infty$ which form a basis of $\cD$.
The eigenfunctions can be chosen to be orthogonal and normalized such that $\norm{\phi_i}_{L^2(\pi)} = 1$; we always assume this is the case.
Then, $\cL$ can be expressed as
\[
\cL=\sum_{i=1}^\infty \lambda_i \, \langle\phi_i,\cdot\rangle_{L^2(\pi)} \, \phi_i.
\]

The operator $\cL$ has a discrete spectrum under the following condition (\cite{friedrichsSpektraltheorieHalbbeschrankterOperatoren1934},~\cite[Theorem XIII.67]{reedsimon1978vol4},~\cite[Corollary 4.10.9]{bakry2014markov}):
\begin{align*}
    V_{\msf S} \in L^1_{\rm loc}(\R^d), \qquad \inf V_{\msf S} > -\infty, \qquad\text{and}\qquad \lim_{\norm x\to\infty} V_{\msf S}(x) = +\infty,
\end{align*}
where $V_{\msf S} := -\Delta V + \frac{1}{2} \norm{\nabla V}^2$.
Moreover, under this condition we also have $\lambda_i\to\infty$ as $i\to+\infty$.
For example, this condition is satisfied for $V(x)=\norm{x}^\alpha$ for $\alpha>1$, but not for $\alpha=1$.
In fact, for $\alpha=1$, the spectrum of $\cL$ is not discrete \cite[\S 4.1.1]{bakry2014markov}.



The \emph{Poincar\'e inequality}~\eqref{eq:poincare} is interpreted as a \emph{spectral gap inequality}, since it asserts that $\lambda_1 = 1/C_{\msf P} > 0$.
Thus, under a Poincar\'e inequality, $\cL : \cD\cap\{f\in L^2(\pi) \mid \E_\pi f = 0\}\to L^2(\pi)$ is bijective. Moreover, if it has a discrete spectrum, its inverse satisfies
\[
\cL^{-1}=\sum_{i=1}^\infty \lambda_i^{-1}\, \langle\phi_i,\cdot\rangle_{L^2(\pi)} \, \phi_i.
\]

\section{Proofs of the convergence guarantees for the chi-squared flow}\label{sec:csf_proofs}

In this section, we give proofs of the convergence results we stated in Section~\ref{sec:csf}.

\begin{proof}[Proof of Theorem~\ref{thm:csf_chi_conv_poincare} (non-log-concave case)]
According to~\cite[Proposition 1]{chewi2020exponential}, the Poincar\'e inequality implies the following inequality for the chi-squared divergence:
\begin{align}
\label{eq:loja}
    {\chi^2(\mu \mmid \pi)}^{3/2} \le \frac{9C_{\msf P}}{4} \E_\mu\bigl[\bigl\lVert \nabla \frac{\D\mu}{\D\pi}\bigr\rVert^2\bigr], \qquad \forall \mu \ll \pi.
\end{align}
Since the Wasserstein gradient of $\chi^2(\cdot\mmid\pi)$ at $\mu$ is given by $2\nabla (\D \mu/\D \pi)$ (see Section~\ref{sec:wgf}), it yields
\begin{align*}
    \partial_t \chi^2(\mu_t \mmid \pi)
    &= -4\E_{\mu_t}\bigl[\bigl\lVert \nabla \frac{\D\mu_t}{\D\pi}\bigr\rVert^2] \le -\frac{16}{9C_{\msf P}} {\chi^2(\mu_t \mmid \pi)}^{3/2}
\end{align*}
Solving the above differential inequality yields
\begin{align*}
    \chi^2(\mu_t \mmid \pi)
    &\le \frac{\chi^2(\mu_0\mmid \pi)}{{\{1 + 8t\sqrt{\chi^2(\mu_0\mmid \pi)}/(9C_{\msf P}) \}}^2},
\end{align*}
which implies the desired result.
\end{proof}

We now prepare for the proof of exponentially fast convergence in chi-squared divergence for log-concave measures. The key to proving such results lies in differential inequalities of the form
\begin{align}\label{eq:chi_sq_pl}
    \chi^2(\mu \mmid \pi)
    &\le C_{\sf PL} \E_\mu\bigl[\bigl\lVert \nabla \frac{\D\mu}{\D\pi} \bigr\rVert^2\bigr], \qquad \forall \mu \ll \pi,
\end{align}
which may be interpreted as a \emph{Polyak-\L{}ojasiewicz \emph{(PL)} inequality}~\cite{karimi2016linear} for {$\chi^2(\cdot\mmid \pi)$}. 
PL inequalities are well-known in the optimization literature, and can be even used when the objective is not convex~\cite{CheMauRig20}.
In contrast, the preceding proof uses the weaker inequality~\eqref{eq:loja},
which may be interpreted as a \emph{\L{}ojasiewicz inequality}~\cite{lojasiewicz1963propriete}.

To see that a PL inequality readily yields exponential convergence, observe that
\begin{align*}
    \partial_t \chi^2(\mu_t \mmid \pi)
    &= -4\E_{\mu_t}\bigl[\bigl\lVert \nabla \frac{\D\mu_t}{\D\pi}\bigr\rVert^2] \le -\frac{4}{C_{\sf PL}} \chi^2(\mu_t \mmid \pi)\,.
\end{align*}
Together with Gr\"onwall's inequality, the differential inequality yields $\chi^2(\mu_t \mmid \pi) \le \chi^2(\mu_0 \mmid \pi) \, e^{-\frac{4t}{C_{\sf PL}}}$.

In order to prove a PL inequality of the type~\eqref{eq:chi_sq_pl}, we require two ingredients.
The first one is a transportation-cost inequality for the chi-squared divergence proven in~\cite{liu2020quadratictransport}, building on the works~\cite{ding2015quadratictransport, ledoux2018remarks}. It asserts that if $\pi$ satisfies a Poincar\'e inequality~\eqref{eq:poincare}, then the following inequality holds:
\begin{align}\label{eq:chi_sq_transport}
    W_2^2(\mu, \pi)
    &\le 2C_{\msf P} \chi^2(\mu \mmid \pi), \qquad\forall \mu \ll \pi.
\end{align}
For the second ingredient, we use an argument of~\cite{otto2000generalization} to show that if $\pi$ satisfies a chi-squared transportation-cost inequality such as~\eqref{eq:chi_sq_transport}, and in addition is log-concave, then it satisfies an inequality of the type~\eqref{eq:chi_sq_pl}. We remark that the converse statement, that is, if $\pi$ satisfies a PL inequality~\eqref{eq:chi_sq_pl} then it satisfies an appropriate chi-squared transportation-cost inequality, was proven in~\cite{chewi2020exponential} without the additional assumption of log-concavity. It implies that for log-concave distributions, the PL inequality~\eqref{eq:chi_sq_pl} and the 
chi-squared transportation-cost inequality~\eqref{eq:chi_sq_transport} are, in fact, equivalent.

\begin{thm}\label{thm:chi2transport2Poincare}
Let $\pi$ be log-concave, and assume that for some $q \in (1,\infty)$ and a constant ${\sf C} > 0$,
    \[
    W_2^2(\mu,\pi)\le {\sf C} {\chi^2(\mu\mmid\pi)}^{2/q}, \qquad \forall \ \mu \ll \pi.
    \]
    Then,
    \begin{equation}\label{eq:chi2LSI}
    {\chi^2(\mu\mmid \pi)}^{2/p}\le 4{\sf C}\E_\mu\bigl[\bigl\lVert \nabla\frac{\D\mu}{\D\pi}\bigr\rVert^2\bigr], \qquad \forall \ \mu \ll \pi,
    \end{equation}
    where $p$ satisfies $1/p + 1/q = 1$.
\end{thm}
\begin{proof}
Following~\cite{otto2000generalization}, let $T$ be the optimal transport map from $\mu$ to $\pi$.
Since $\chi^2(\cdot\mmid\pi)$ is displacement convex~\cite{ohta2011generalizedentropies, ohta2013generalizedentropiesii} and has Wasserstein gradient $2\nabla(\D \mu/\D\pi)$ at $\mu$ (c.f.\ Section~\ref{sec:wgf}), the ``above-tangent'' formulation of displacement convexity (\cite[Proposition 5.29]{villani2003topics}) yields
\begin{align*}
    0
    &= \chi^2(\pi \mmid \pi)
    \ge \chi^2(\mu \mmid \pi) + 2\E_\mu\bigl\langle \nabla \frac{\D\mu}{\D\pi}, T - \id\bigr\rangle
    \ge \chi^2(\mu \mmid \pi) - 2W_2(\mu,\pi) \sqrt{\E_\mu\bigl[\bigl\lVert \nabla \frac{\D\mu}{\D\pi}\bigr\rVert^2 \bigr]},
\end{align*}
where we used the Cauchy-Schwarz inequality for the last inequality.
Rearranging the above display and using the  transportation-cost inequality assumed in the statement of theorem, we get
\begin{align*}
    \chi^2(\mu\mmid \pi)
    &\le 2W_2(\mu,\pi) \sqrt{\E_\mu\bigl[\bigl\lVert \nabla \frac{\D\mu}{\D\pi}\bigr\rVert^2 \bigr]}
    \le 2\sqrt{{\sf C}\E_\mu\bigl[\bigl\lVert \nabla \frac{\D\mu}{\D\pi}\bigr\rVert^2 \bigr]} \, {\chi^2(\mu\mmid \pi)}^{1/q}.
\end{align*}
The result follows by rearranging the terms.
\end{proof}

\begin{proof}[Proof of Theorem~\ref{thm:csf_chi_conv_poincare} (log-concave case)]
    From the transportation-cost inequality~\eqref{eq:chi_sq_transport} and Theorem~\ref{thm:chi2transport2Poincare} with $p=q=2$, we obtain
    \begin{align*}
        \chi^2(\mu\mmid\pi)
        &\le 8C_{\msf P} \E_\mu\bigl[\bigl\lVert \nabla \frac{\D\mu}{\D\pi} \bigr\rVert^2].
    \end{align*}
This PL inequality together with Gr\"onwall's inequality readily yields the result.
\end{proof}

We conclude this section with the proof of Theorem~\ref{thm:csf_conv_lsi}, which shows exponential convergence of \ref{eq:csf} in chi-squared divergence under the assumption of a log-Sobolev inequality~\eqref{eq:lsi} (but without the assumption of log-concavity).

\begin{proof}[Proof of Theorem~\ref{thm:csf_conv_lsi}]
We first claim that
    \begin{align}\label{eq:deriv_chi_sq_lsi}
        \partial_t \chi^2(\mu_t\mmid \pi)
        &\le - \frac{4}{9C_{\msf{LSI}}} {[\chi^2(\mu_t\mmid \pi) +1]}^{3/2} \ln[\chi^2(\mu_t\mmid \pi) + 1].
    \end{align}
    Indeed, applying~\eqref{eq:lsi}, we obtain
    \begin{align*}
        \partial_t \chi^2(\mu_t\mmid \pi)
        &= - 4 \int \bigl\lVert \nabla \frac{\D\mu_t}{\D\pi} \bigr\rVert^2 \, \D \mu_t
        = - \frac{16}{9} \int \bigl\lVert \nabla \bigl\lvert \frac{\D\mu_t}{\D\pi} \bigr\rvert^{3/2} \bigr\rVert^2 \, \D \pi
        \le - \frac{8}{9C_{\msf{LSI}}} \on{ent}_\pi\bigl(\bigl\lvert \frac{\D\mu_t}{\D\pi} \bigr\rvert^3\bigr).
    \end{align*}
    Next, the variational formula for the entropy gives
    \begin{align*}
        \on{ent}_\pi f
        &= \sup\{\E_\pi(fg) : g~\text{satisfies}~\E_\pi\exp g = 1\},
    \end{align*}
    see~\cite[Lemma 3.15]{vanhandelProbabilityHighDimension2014} or~\cite[Theorem 4.13]{boucheron2013concentration}.
 Choosing $g = \ln(\D\mu_t/\D\pi)$ yields
    \begin{align*}
        \on{ent}_\pi\bigl(\bigl\lvert \frac{\D\mu_t}{\D\pi} \bigr\rvert^3\bigr)
        &\ge \E_\pi\bigl[\bigl\lvert \frac{\D\mu_t}{\D\pi} \bigr\rvert^3 \ln \frac{\D\mu_t}{\D\pi} \bigr]
        = \frac{1}{3} \E_\pi\bigl[\bigl\lvert \frac{\D\mu_t}{\D\pi} \bigr\rvert^3 \ln\bigl( \bigl\lvert \frac{\D\mu_t}{\D\pi} \bigr\rvert^3\bigr)\bigr] \\
        &\ge \frac{1}{3} \E_\pi\bigl[ \bigl\lvert \frac{\D\mu_t}{\D\pi} \bigr\rvert^3 \bigr] \ln \E_\pi\bigl[ \bigl\lvert \frac{\D\mu_t}{\D\pi} \bigr\rvert^3\bigr]\\
        &\ge \frac{1}{2} \E_\pi\bigl[ \bigl\lvert \frac{\D\mu_t}{\D\pi} \bigr\rvert^2 \bigr]^{3/2} \ln \E_\pi\bigl[ \bigl\lvert \frac{\D\mu_t}{\D\pi} \bigr\rvert^2\bigr]\\
        &= \frac{1}{2} {[\chi^2(\mu_t\mmid\pi) + 1]}^{3/2} \ln[\chi^2(\mu_t\mmid \pi) + 1],
    \end{align*}
    where in the second inequality, we used  that $x\mapsto x\ln x$ is convex on $\R_+$ and in the third, we used that it increasing when $x\ge 1$ together with 
    \[
    \E_\pi\bigl[ \bigl\lvert \frac{\D\mu_t}{\D\pi} \bigr\rvert^2\bigr] = 1 +\chi^2(\mu_t\mmid\pi)  \ge 1.
    \]
    This proves~\eqref{eq:deriv_chi_sq_lsi}.
    
    To simplify the inequality~\eqref{eq:deriv_chi_sq_lsi}, we use the crude bounds
    \[
    \ln[\chi^2(\mu_t\mmid\pi) + 1] \ge \begin{cases}
     1, & \text{if}~\chi^2(\mu_t\mmid\pi) \ge \e - 1\\
     \chi^2(\mu_t\mmid\pi)/2,& \text{otherwise}.
    \end{cases}
    \]
It yields respectively
    \begin{align}\label{eq:chi_sq_lsi}
            \partial_t \chi^2(\mu_t\mmid \pi) \le - \frac{2}{9C_{\msf{LSI}}} \begin{cases}
            2 {\chi^2(\mu_t\mmid\pi)}^{3/2}, & \text{if}~\chi^2(\mu_t\mmid\pi) \ge \e - 1, \\
            \chi^2(\mu_t\mmid\pi), &\text{otherwise.}
        \end{cases}
    \end{align}
    
Solving the differential inequality in the first case yields
    \begin{align*}
      e-1\le   \chi^2(\mu_t\mmid\pi)
        &\le \Bigl[ \frac{9C_{\msf{LSI}} \sqrt{\chi^2(\mu_0\mmid\pi)}}{9C_{\msf{LSI}} + 2t\sqrt{\chi^2(\mu_0 \mmid \pi)}} \Bigr]^2 \le \Bigl[ \frac{9C_{\msf{LSI}}}{2t} \Bigr]^2,
    \end{align*}
    so that in this first case, it must holds that
    \begin{align*}
        t \le \frac{9C_{\msf{LSI}}}{2 \sqrt{\e - 1}}
        < 3.5C_{\msf{LSI}}=:t_0.
    \end{align*}
Therefore, if $t \ge t_0$, we are in the second case. In particular, $ \chi^2(\mu_{t_0}\mmid \pi) \le e-1 \le 2$ and integrating the differential inequality between $t_0$ and $t$ we get
    $$
 \chi^2(\mu_t\mmid \pi) \le  \chi^2(\mu_{t_0}\mmid \pi) \, e^{-\frac{2(t-t_0)}{9C_{\sf LSI}}} \le  \big(\chi^2(\mu_{0}\mmid \pi) \wedge 2\big) \, e^{-\frac{2(t-t_0)}{9C_{\sf LSI}}}\,,
    $$
    where in the last inequality, we used the fact that $t \mapsto \chi^2(\mu_{t}\mmid \pi) $ is decreasing so that it also holds $\chi^2(\mu_{t_0}\mmid \pi) \le \chi^2(\mu_{0}\mmid \pi)$.
    In particular, taking $t \ge 2t_0=7C_{\msf{LSI}}$ yields the desired result.
\end{proof}




\section{Details for the experiments}\label{append:numerics}

We give additional details for the experiments presented in this paper. All methods were implemented in \texttt{Python}. Since the Schr\"odinger operator requires the Laplacian and gradient of the potential $V$, we employ automatic differentiation to avoid laborious calculations of these derivatives.

The \emph{probabilists'} Hermite polynomials are well-known to be eigenfunctions of the 1D Ornstein-Uhlenbeck operator $\ms L$ given by $\ms Lf(x) := -f''(x) + xf'(x)$, and they satisfy the recursive relationship $H_{n+1}(x) = xH_n(x) -  nH_{n-1}(x)$, with $H_0(x) = 1$ and $H_1(x) = x$. It also holds that $H_n'(x) = n H_{n-1}(x)$. With these equations, it is easy to check that the eigenvalue corresponding to $H_n$ is $\lambda_n = n$. These are used as the eigenfunctions and eigenvalues in the standard normal example given in Figure~\ref{fig:gaussian}. In the simulation, we use the first 150 Hermite polynomials. We run \ref{eq:msvgd} for 2000 iterations with a constant step size, with initial points drawn uniformly from the interval $[2.5, 4.5]$.

In Figure~\ref{fig:gaussmix2d}, we display an example of sampling 50 particles from a mixture of two 2-dimensional Gaussian distributions given by $\pi = \frac{1}{2} \mc N((-1,-1)^\top, I_2) + \frac{1}{2} \mc N((1,1)^\top, I_2)$. 
To run this experiment, we use a 2-dimensional FD method, which approximates the Laplacian as 
$$\Delta_\varepsilon f(x,y) := \frac{f(x-\varepsilon, y) + f(x+\varepsilon, y) + f(x, y - \varepsilon) + f(x, y+\varepsilon) - 4f(x)}{\varepsilon^2}.$$
We again use the Schr\"odinger operator for stability and use FD again to compute the gradients of the eigenfunctions. We use a $128\times 128$ grid of evenly spaced $x$ and $y$ values between $-6$ and $6$. We calculate only the bottom 100 eigenvalues and eigenfunctions,  since the other eigenfunctions incur additional computational cost without noticeably changing the result. Any negative eigenvalues (which arise from numerical errors) are discarded.

\begin{figure}[H]
    \centering
    \includegraphics[width = 0.31\textwidth]{trajectories_law_2d.pdf}
    \includegraphics[width = 0.31\textwidth]{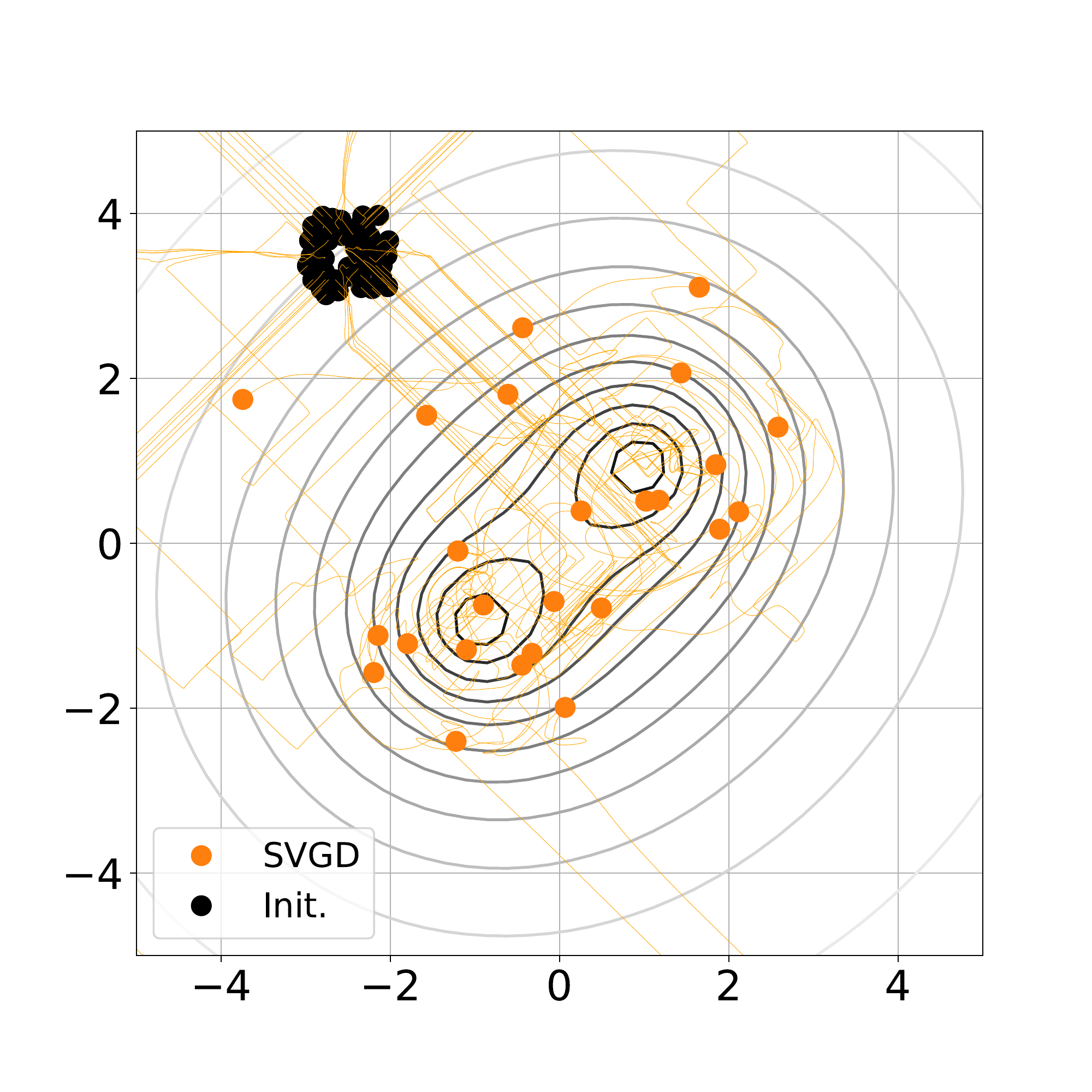}
    \includegraphics[width = 0.31\textwidth]{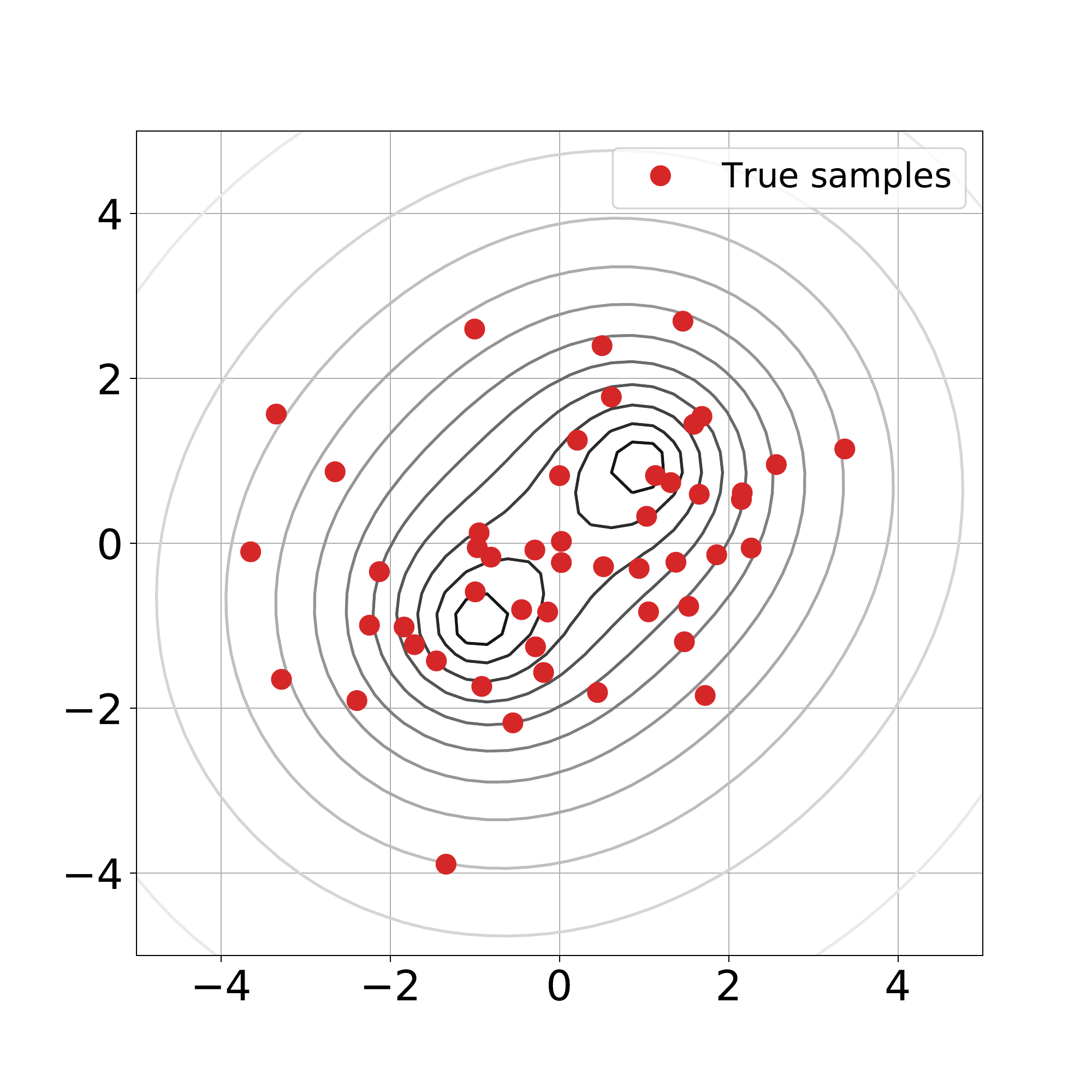}
    \caption{Left: 50 particles and trajectories generated from $\frac{1}{2} \mc N((-1,-1)^\top, I_2) + \frac{1}{2} \mc N((1,1)^\top, I_2)$ with \ref{eq:msvgd}. Middle: 50 particles and trajectories generated by SVGD. Right: true samples from the distribution.}
    \label{fig:gaussmix2d_2}
\end{figure}

\begin{figure}[h]
    \centering
    \includegraphics[width=0.23\textwidth]{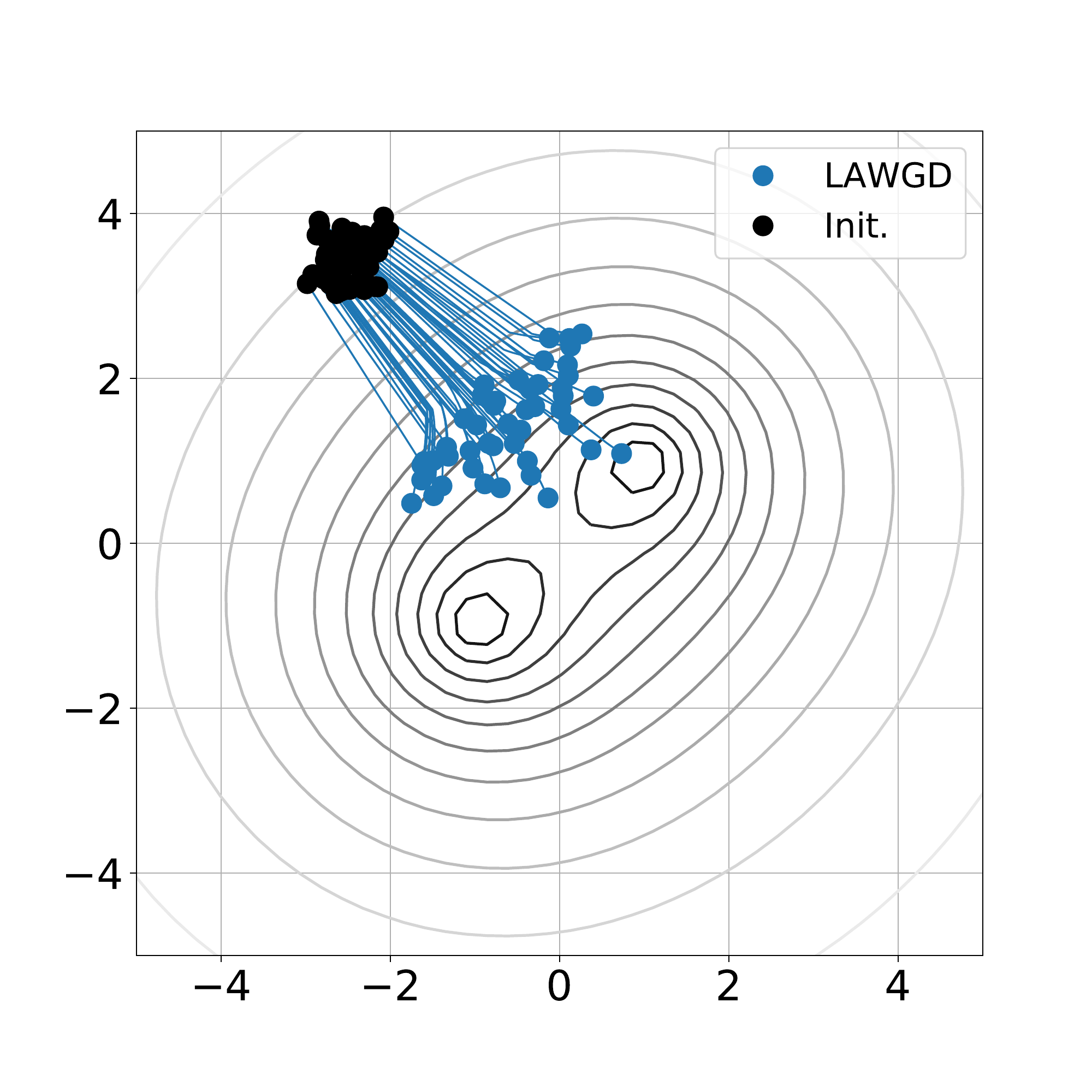}
    \includegraphics[width=0.23\textwidth]{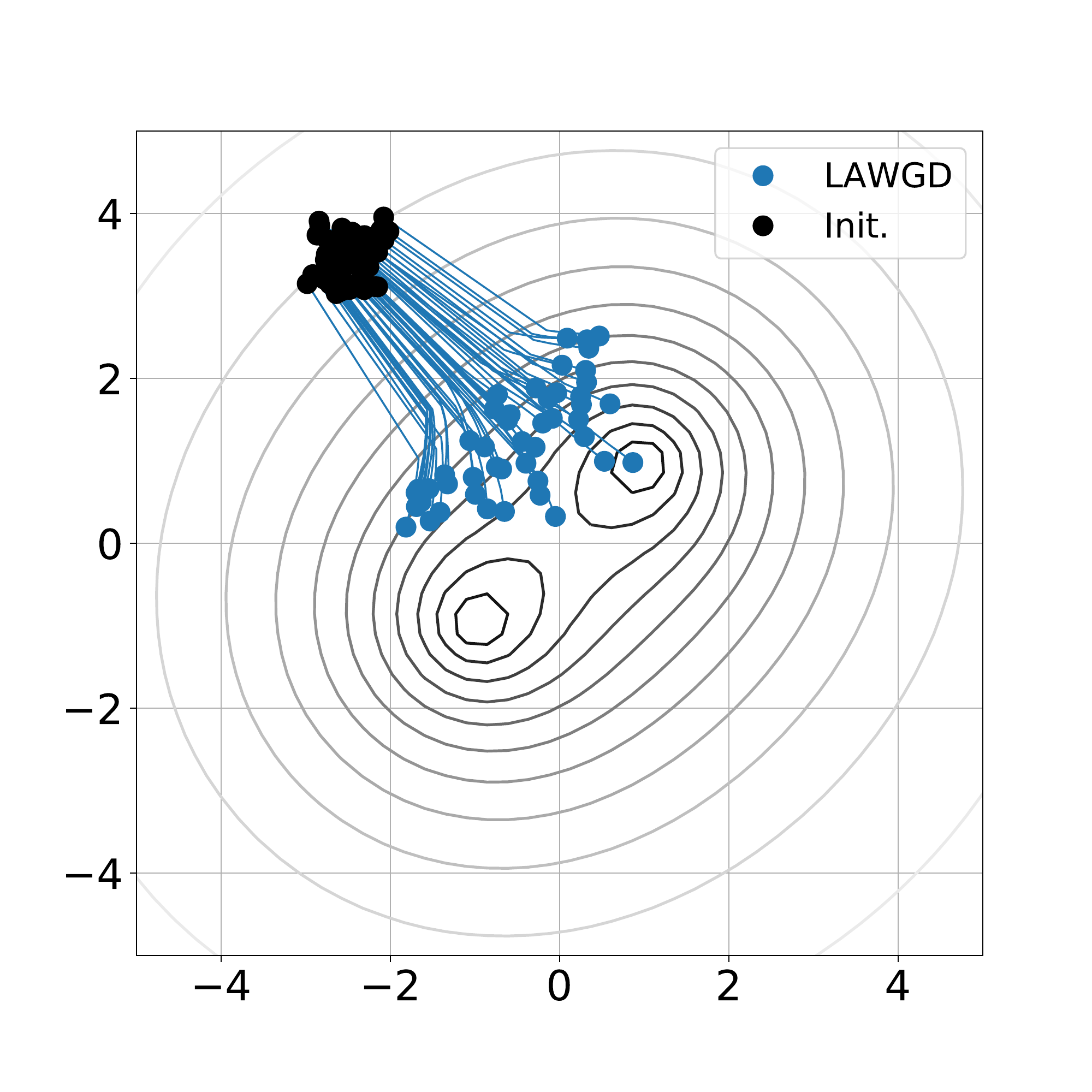}
    \includegraphics[width=0.23\textwidth]{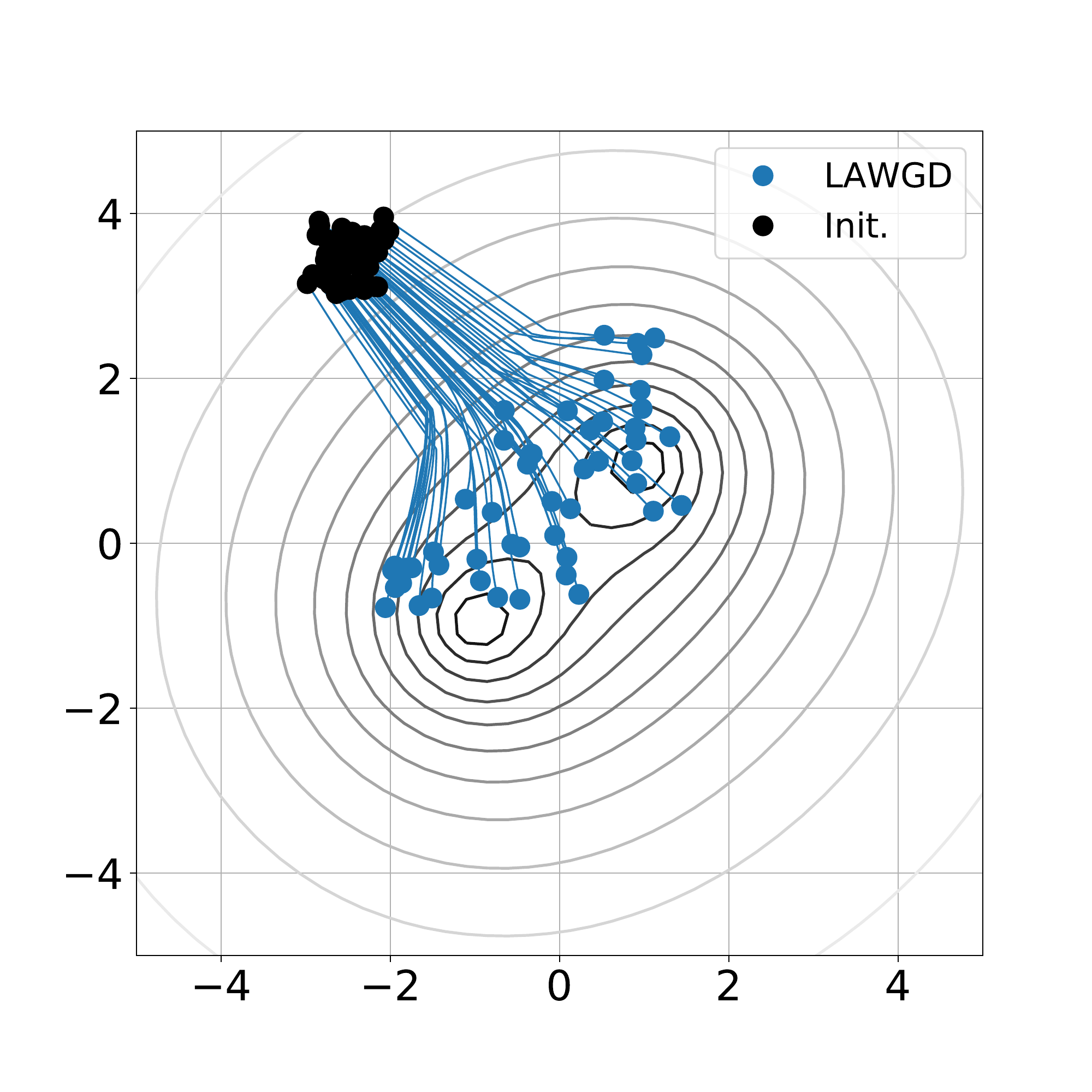}
    \includegraphics[width=0.23\textwidth]{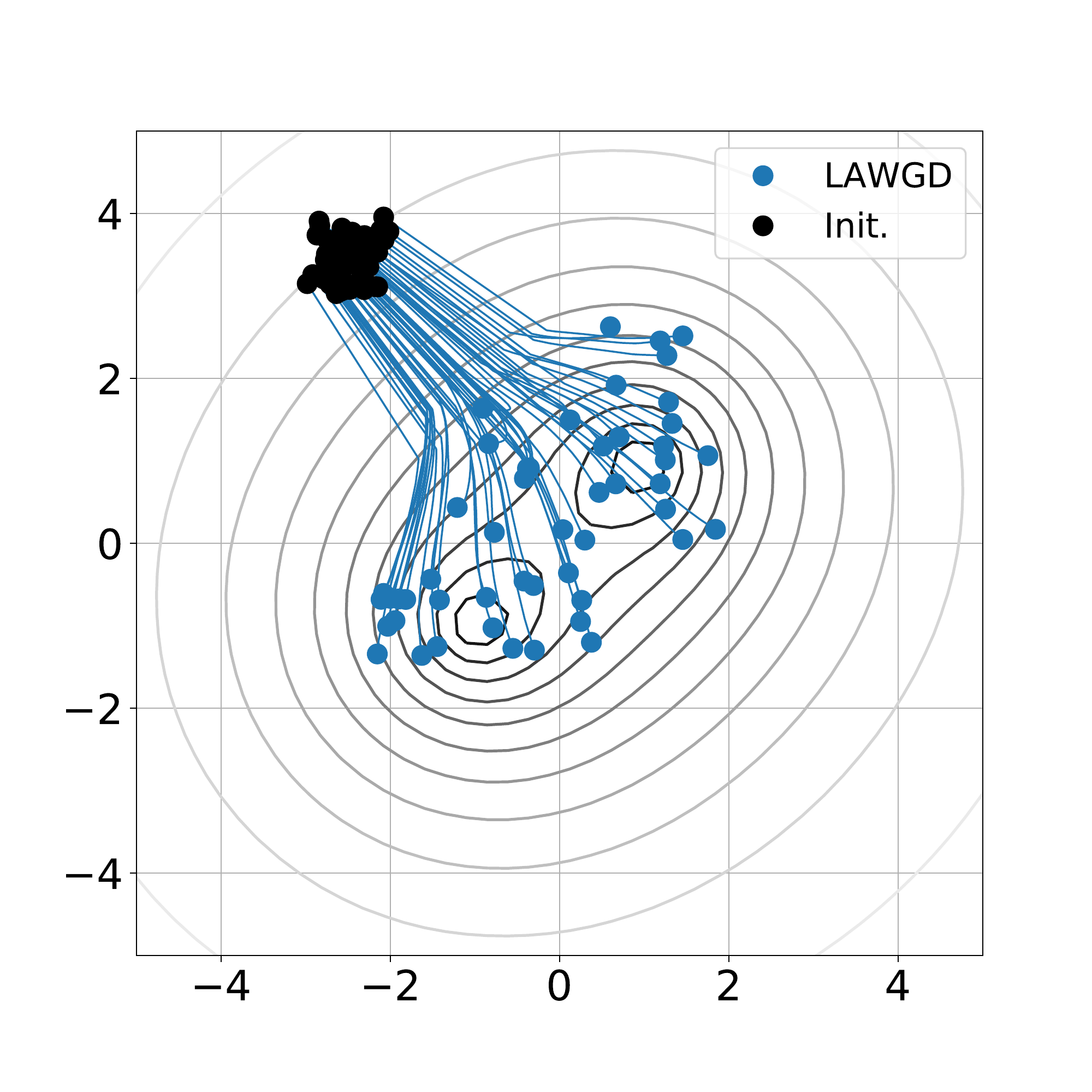}
    \includegraphics[width=0.23\textwidth]{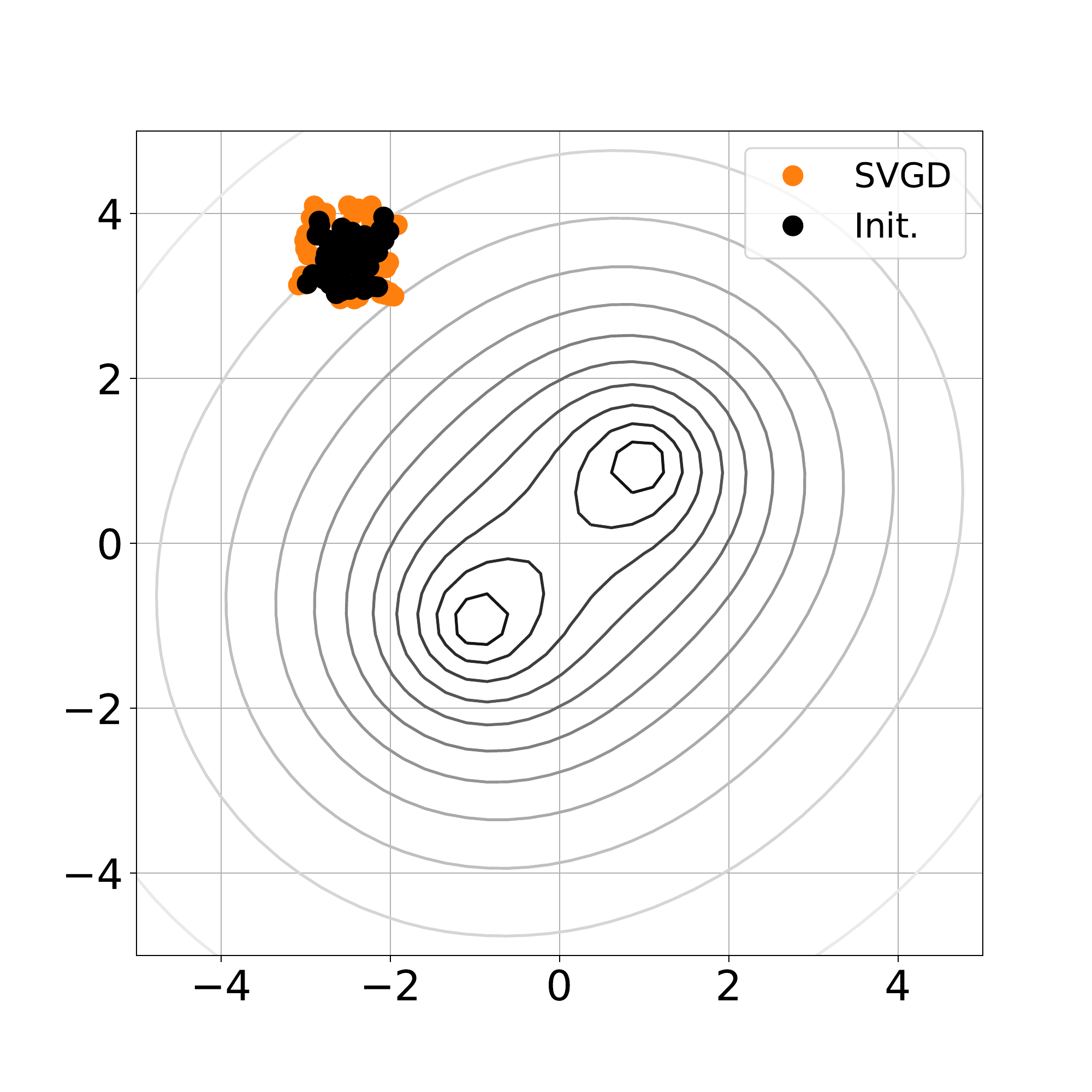}
    \includegraphics[width=0.23\textwidth]{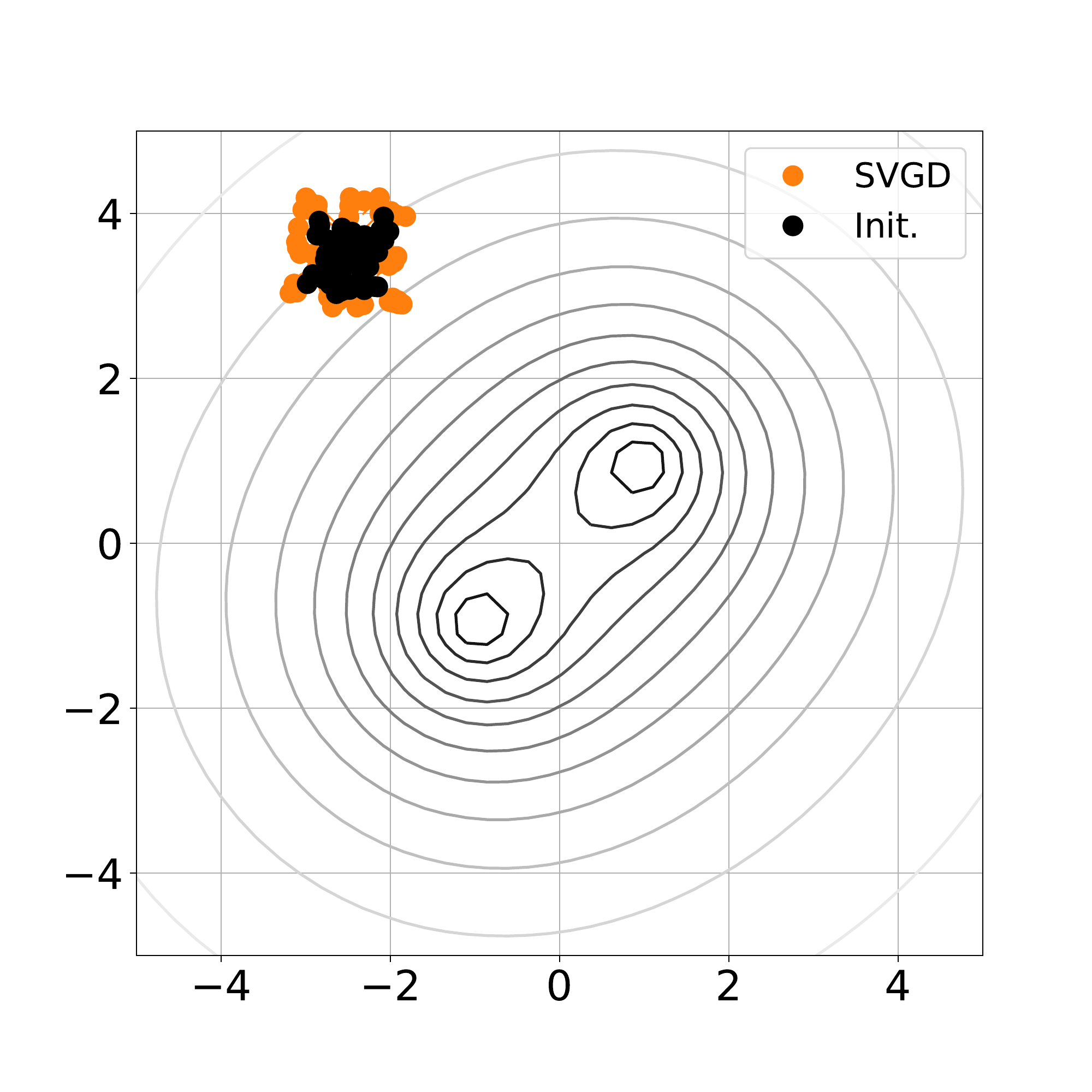}
    \includegraphics[width=0.23\textwidth]{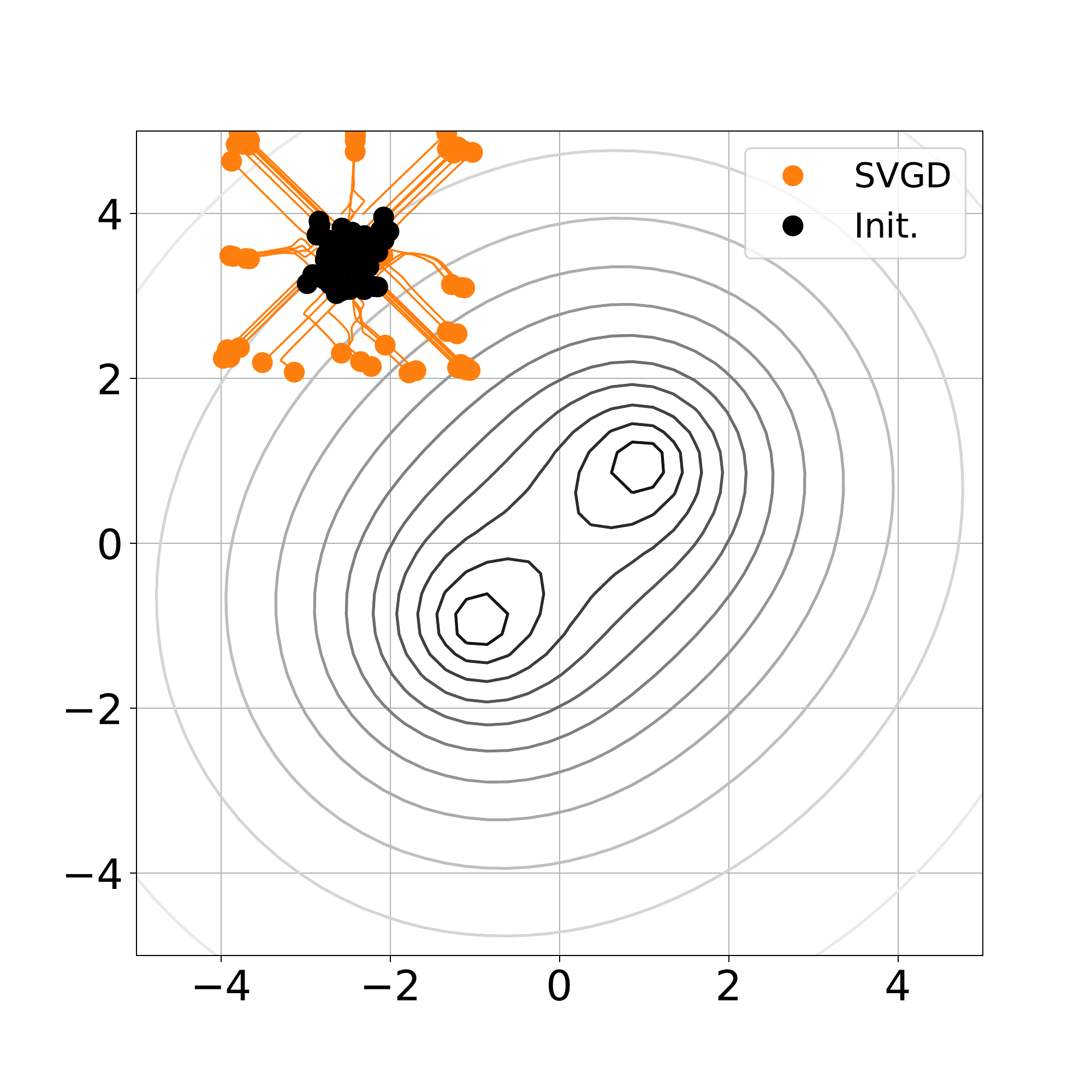}
    \includegraphics[width=0.23\textwidth]{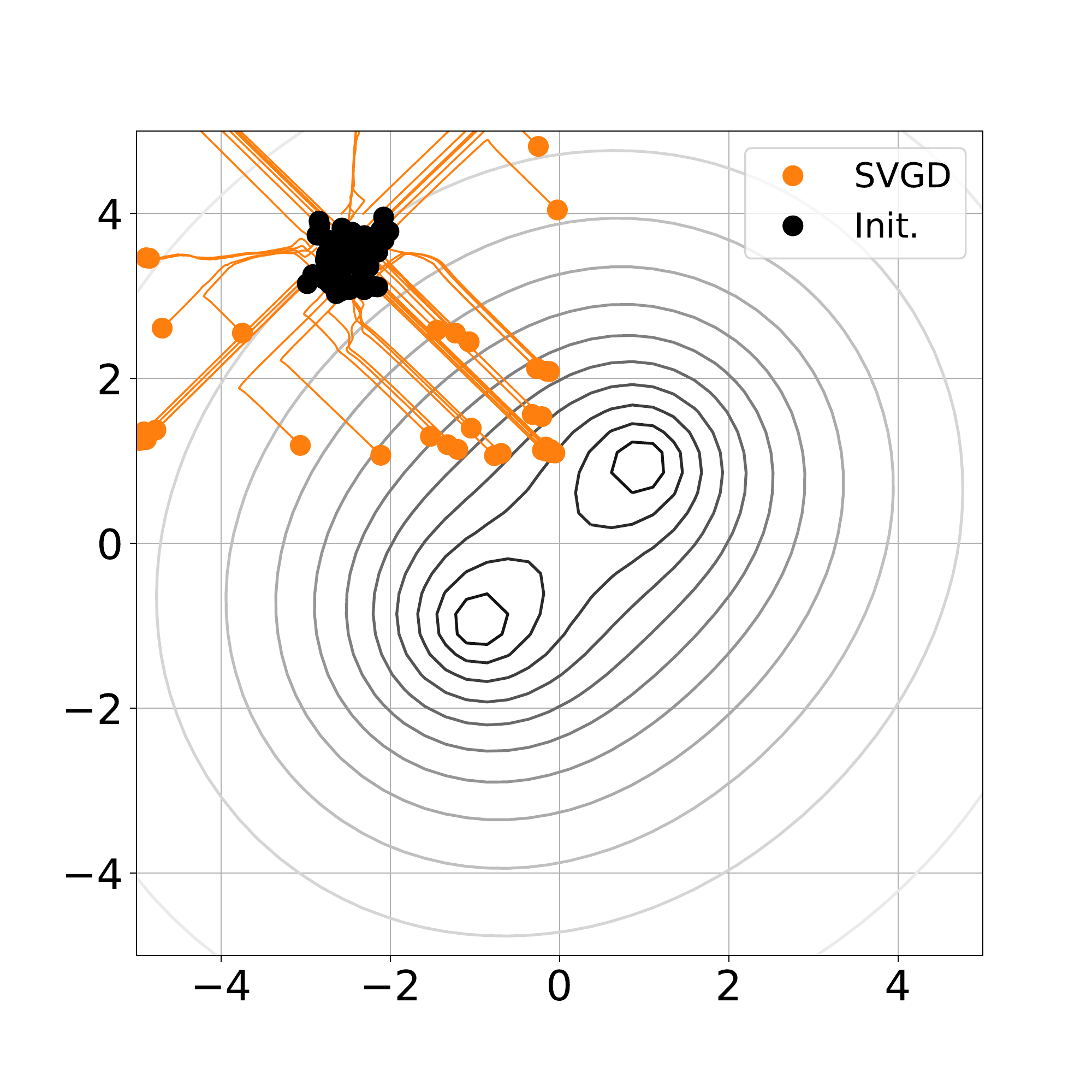}
    \caption{Top: \ref{eq:msvgd} after 100, 200, 1000, and 2000 iterations. Bottom: SVGD after 100, 200, 1000, and 2000 iterations.}
    \label{fig:lawgd_over_time}
\end{figure}

Additionally, we display the results from running SVGD with the RBF kernel and median-based bandwith on this example with a less favorable initialization. True samples from $\pi$ are displayed for comparison.  Both \ref{eq:msvgd} and SVGD are run for 20000 iterations with a constant step size. The samples from \ref{eq:msvgd} tend to move very fast from their initial positions and then tend to settle into their final positions as seen in Figure~\ref{fig:gaussmix2d_2}. On the other hand, with constant step size, the samples of SVGD do not seem to converge, and one must use a decreasing step size scheme in order for the particles to stabilize. We also note that many of the samples generated by SVGD tend to blow up with a constant step size.

In Figure~\ref{fig:lawgd_over_time}, we plot the particles of \ref{eq:msvgd} and SVGD at iterations 100, 200, 1000, and 2000 to compare the speed of convergence.

\bibliographystyle{aomalpha}
\bibliography{ref}

\end{document}